\documentclass{patmorin}
\usepackage{pat}
\usepackage[T1]{fontenc}
\usepackage[utf8]{inputenc}
\usepackage{mathtools}
\usepackage{thmtools}
\usepackage{thm-restate}
\usepackage{siunitx}
\usepackage{comment}
\usepackage[normalem]{ulem}
\usepackage[svgnames, dvipsnames]{xcolor}

\usepackage{tikz}
\usetikzlibrary{hobby, decorations.markings, calc}

\definecolor{turfgreen}{HTML}{266F57}
\definecolor{goldenorange}{HTML}{E59500}
\definecolor{steelblue}{rgb}{0.27, 0.51, 0.71}

\newcommand{\Ucol}{turfgreen}
\newcommand{\Mcol}{brightmaroon}
\newcommand{\Ecol}{goldenorange}
\newcommand{\projcol}{brightmaroon} % \newcommand{\projcol}{steelblue}

\makeatletter
\newcommand*{\greek}[1]{%
  \expandafter\@greek\csname c@#1\endcsname
}
\newcommand*{\@greek}[1]{%
  $\ifcase#1\or\alpha\or\beta\or\gamma\or\delta\or\varepsilon
    \or\zeta\or\eta\or\theta\or\iota\or\kappa\or\lambda
    \or\mu\or\nu\or\xi\or o\or\pi\or\varrho\or\sigma
    \or\tau\or\upsilon\or\phi\or\chi\or\psi\or\omega
    \else\@ctrerr\fi$
}
\makeatother

\usepackage{enumerate}
\usepackage[shortlabels,inline]{enumitem}
\renewenvironment{enumerate}{\begin{enumorig}[label=\textup{(\roman*)}, noitemsep, 
topsep=2pt plus 2pt, labelindent=.2em, leftmargin=*, widest=iii]}{\end{enumorig}}

\newenvironment{enumeratealph}{\begin{enumorig}[label=\textup{(\alph*)}, noitemsep, 
topsep=2pt plus 2pt, labelindent=.2em, leftmargin=*, widest=10]}{\end{enumorig}}

\renewenvironment{itemize}{\begin{itemorig}[label=\tiny\textbullet, noitemsep, 
topsep=2pt plus 2pt, labelindent=.2em, leftmargin=*, widest=ii]}{\end{itemorig}}

\iftrue    % change to \iffalse to skip comments 
\newcommand{\vida}[1]{{\color{Green}Vida: #1}}
\newcommand{\pat}[1]{\textcolor{Blue}{Pat: #1}}
\newcommand{\gwen}[1]{\textcolor{Purple}{Gwen: #1}}
\newcommand{\piotr}[1]{\textcolor{red}{Piotr: #1}}
\newcommand{\raj}[1]{\textcolor{orange}{Raj: #1}}
\else
\newcommand{\vida}[1]{}
\newcommand{\pat}[1]{}
\newcommand{\gwen}[1]{}
\newcommand{\piotr}[1]{}
\newcommand{\raj}[1]{}
\fi

\crefname{p}{}{}
\creflabelformat{p}{#2(#1)#3}

\newenvironment{clmproof}{\begin{proof}[Proof of Claim]}{\end{proof}}

\usepackage[longnamesfirst,numbers,sort&compress]{natbib}

\DeclareMathOperator{\len}{len}
\DeclareMathOperator{\tw}{tw}

\DeclareMathOperator{\pad}{pad}
\DeclareMathOperator{\Aux}{Aux}

\DeclareMathOperator{\proj}{proj}
\DeclareMathOperator{\spn}{span}

\DeclareMathOperator{\dist}{dist}

\DeclareMathOperator{\polylog}{polylog}

\newcommand{\NN}{\mathbb{N}}

\newcommand{\Oh}{\mathcal{O}}

\DeclarePairedDelimiter\set{\{}{\}}

\let\angle\relax
\DeclarePairedDelimiter\angle{{\!}\langle}{\rangle}

\renewcommand{\ge}{\geqslant}
\renewcommand{\le}{\leqslant}
\renewcommand{\geq}{\geqslant}
\renewcommand{\leq}{\leqslant}

\renewcommand{\emptyset}{\varnothing}

\title{\MakeUppercase{Far-apart {E}rdős--{P}ósa property of long cycles}}
\author{%
 Maria Chudnovsky\,\thanks{Princeton University, Princeton, NJ 08544, USA (\texttt{mchudnov@math.princeton.edu}).
 Supported by NSF Grants DMS-2348219 and CCF-2505100, AFOSR grant FA9550-25-1-0275, and a Guggenheim Fellowship.}
 \qquad
 Vida Dujmovi{\'c}\,\thanks{School of Computer Science and Electrical Engineering, University of Ottawa, Ottawa, Canada (\texttt{vida.dujmovic@uottawa.ca}). Research supported by NSERC and a University of Ottawa Research Chair.}
 \qquad
 Gwena\"el Joret\thanks{D\'epartement d'Informatique, Universit\'e libre de Bruxelles, Belgium ({\tt gwenael.joret@ulb.be}). G.\ Joret is supported by the Belgian National Fund for Scientific Research (FNRS).}
 \qquad
Raj Kaul\thanks{School of Mathematics, Monash University, Australia ({\tt raj.kaul@monash.edu}). Research supported by an Australian Government Research Training Program Scholarship.}
 \qquad
Piotr Micek\thanks{Department of Theoretical Computer Science, Jagiellonian University, Kraków, Poland (\texttt{piotr.micek@uj.edu.pl}). Research supported by the National Science Center of Poland under grant UMO-2023/05/Y/ST6/00079 within the WEAVE-UNISONO program.}
 \qquad
 Pat Morin\thanks{School of Computer Science, Carleton University, Ottawa, Canada (\texttt{morin@scs.carleton.ca}). Research supported by NSERC.}
  \qquad
 Alex Scott\,\thanks{Mathematical Institute, University of Oxford, Oxford, UK (\texttt{alexander.scott@maths.ox.ac.uk}). Research supported by EPSRC grant EP/X013642/1}
 \qquad%
}
\date{}

\begin{document}

\maketitle

\begin{abstract}
We prove that there exist functions $f:\mathbb N^2\to\mathbb N$ and $g:\mathbb N\to\mathbb N$ such that for all positive integers $k$, $d$, and $\ell\ge3$, every graph $G$ either contains $k$ cycles of length at least $\ell$ that are pairwise at distance greater than $d$, or admits a subset of vertices $X$ with $|X|\le f(k,\ell)$ such that $G-B_G(X,g(d))$ contains no cycle of length at least $\ell$, where $B_G(X,r)$ denotes the ball of radius $r$ around $X$.
This generalizes a theorem of Dujmović, Joret, Micek, and Morin (2024), which established the $\ell=3$ case. 
Moreover, we prove that the theorem holds with $f(k,\ell)\in\Oh(\ell k\log k)$ and $g(d)\in\Oh(d)$.
The linear bound on $g$ is best possible, while the bound on $f$ is optimal as a function of $k$ for every fixed $\ell$. In particular, for $\ell=3$ our result improves the previous bound of $\Oh(k^{18}\polylog k)$ by Dujmović et al.
\end{abstract}

\section{Introduction}
\label{sec:intro}
In 1965, Erdős and Pósa~\cite{EP1965} proved that there exists a function $f(k)$ such that for every positive integer $k$ and every graph $G$, either $G$ contains $k$ vertex-disjoint cycles, or there is a subset $X$ of vertices of $G$ with $|X|\le f(k)$ such that $G-X$ has no cycles. This celebrated result has become one of the cornerstones of graph theory and has inspired a vast literature on the \emph{Erdős--Pósa property} for a wide range of combinatorial objects.

Many extensions of the Erdős--Pósa theorem impose additional constraints on the cycles under consideration while retaining the same packing-versus-covering paradigm.
Well-known constraints and generalizations include: long cycles \cite{Birmele2007,FioriniLong2014,MOUSSET201721,BruhnLong2018},
where each cycle has length at least $\ell$;  $S$-cycles \cite{KAKIMURA2011378,PONTECORVI20121134,BruhnLong2018,JoosParity2017,KAKIMURA201297}, where the cycles must intersect a prescribed set $S$; modularity constraints \cite{Dejter1988Unboundedness,ThomassenPresence1988,Wollan2011Modularity,KAKIMURA201297,TightEPfunctionForPlanarMinors,rautenbach2001erdos,Thomassen2001EPHighlyConnected,KAWARABAYASHI2006296,JoosParity2017,gollin2025unified},
where the lengths of cycles have remainder $r$ modulo $m$. Another key ingredient in Erdős--Pósa type theorems is the mode of packing, for example, minors \cite{ROBERTSON198692,TightEPfunctionForPlanarMinors,EPWheelMinors,Kloks2002,BIENSTOCK1992163,BonamyJonesSubcubic,FominStrengthening2011,ExForestMinors2013,dujmovic2025tight}, edge-disjointness \cite{DimitrisAn2018,ArchontiaPacking2016,BruhnLong2019,van_Batenburg_2020,BruhnEdge2021}, half-integral packing \cite{MangoesBlueberries,LiuPacking2022,Huynh2019Unified,GollinUnified2024,kakimura2013half}.

\subsection{Far-Apart Erdős--Pósa Property of Cycles}
\label{subsec:far-apart-property}

The following theorem of \citet{ErdosPosaCyclesFarApart}, which we call the ``far-apart Erd\H{o}s--P\'{o}sa property of cycles'', may be viewed as an extension of the Erd\H{o}s--P\'{o}sa theorem where the cycles being "pairwise far apart" is the mode of packing. Here and throughout the paper, $B_G(X,r)$ denotes the set of vertices at distance at most $r$ from $X$ in $G$.

\begin{thm}[\cite{ErdosPosaCyclesFarApart}]
\label{thm:ErdosPosaCyclesFarApart}
There exist functions $f:\N \to \N$ and $g:\N \to \N$, 
%with $g(d)\in O(d)$, 
such that for all positive integers $k$ and $d$, and for every graph $G$, either $G$ contains $k$ cycles that are pairwise at distance greater than $d$ in $G$, or there exists a subset $X$ of vertices of $G$ with $|X|\leq f(k)$ such that $G-B_G(X,g(d))$ has no cycles.
\end{thm}

\cref{thm:ErdosPosaCyclesFarApart} was conjectured in 2024 independently by Chudnovsky and Seymour\footnote{The problem was discussed in March 2024 at the Barbados Graph Theory Workshop held at the Bellairs Research Institute of McGill University.} and by \citet{ahn.gollin:coarse-SODA}. Furthermore, \citet{ahn.gollin:coarse-SODA} solved the cases $k=2$ and arbitrary $d$ as well as $d=1$ and arbitrary $k$. They also pointed out that balls are needed in \cref{thm:ErdosPosaCyclesFarApart}, even for the $d=1$ case, since in complete graphs, no pair of cycles have distance more than $1$, yet an unbounded number of vertices are needed to hit all cycles. 

%The following is our main result, which generalizes  Theorem~\ref{thm:ErdosPosaCyclesFarApart} to long cycles while obtaining the optimal bounds for every fixed $\ell$.

Long cycles, that is, cycles of length at least $\ell\geq 3$, form one of the most
well-studied instances of the Erdős--Pósa property, with a long line of work
in the classical setting of vertex-disjoint packings. A 1988 result of
\citet{ThomassenPresence1988} implies that for every $\ell$, long cycles have
the Erdős--Pósa property: every graph either contains $k$ vertex-disjoint
cycles of length at least $\ell$, or a set of at most $f(k,\ell)$ vertices
meeting all of them. Thomassen's argument gives
$f(k,\ell)\in 2^{\ell^{\Oh(k)}}$. The bound was improved to $\Oh(\ell k^2)$
by \citet*{Birmele2007}, then to $\Oh(\ell k\log k)$ by \citet*{FioriniLong2014},
and finally to the asymptotically optimal $\Theta(\ell k+k\log k)$ by
\citet*{MOUSSET201721}. The following theorem, which is our main result,
generalizes this line of work to the far-apart setting, and equally
generalizes \cref{thm:ErdosPosaCyclesFarApart} to long cycles, with optimal
bounds for every fixed $\ell$.

\begin{thm}
\label{thm:main-in-intro}
There exist functions $f:\mathbb{N}^2\to\mathbb{N}$ and $g:\mathbb{N}\to\mathbb{N}$, 
%with $g(d)\in O(d)$, 
such that for all integers $k$, $d$, and $\ell$ with $k\ge 1$, $d\ge 1$, and $\ell\geq3$, and for every graph $G$,  either $G$ contains $k$ cycles of length at least $\ell$ that are pairwise at distance greater than $d$ in $G$, or there exists a subset $X$ of vertices of $G$ with $|X|\leq f(k,\ell)$ such that  $G-B_G(X,g(d))$ has no cycles of length at least $\ell$.  
Furthermore, this holds with $f(k,\ell) \in \Oh( \ell k \log k)$ and $g(d) \in \Oh(d)$.
\end{thm}

We note that \citet{ahn.gollin:coarse-arxiv} recently proved this result for $d=1$.

Let us now comment on the optimality of the bounds in \cref{thm:main-in-intro}.
First, $f$ necessarily depends on both $k$ and $\ell$. This is witnessed by the `$d=0$ case' of the problem (which is not covered by
\cref{thm:main-in-intro}): there, balls are not needed, and the optimal size of
a hitting set was determined to be $\Theta(\ell k + k\log k)$ by
\citet{MOUSSET201721}. Second, the bound $g(d) \in \Oh(d)$ in \cref{thm:main-in-intro} is best possible since $g(d)\geq d$ is necessary in any such result.\footnote{\textit{Proof.} We show that there is no function $f$ such that \cref{thm:main-in-intro} holds with $f$ and $g(d)=d-1$. Assume for a contradiction that there is a function $f$. Let $s\coloneqq f(k,\ell)$. Construct a graph $G$ as follows: Write $V(K_{2s+1})=[2s+1]$ and let $K_{2s+1}^*$ be the graph obtained from $K_{2s+1}$ by replacing each edge $ij$ with a path $P_{ij}$ between $i$ and $j$, and of length $d$. Let $\set{C_1, \dots, C_{2s+1}}$ be a collection of pairwise vertex-disjoint cycles each of length $\ell$ such that $V(C_i)\cap V(K_{2s+1}^*)=\set{i}$ for all $i\in [2s+1]$. Let $G\coloneqq K_{2s+1}^*\cup C_1\cup \cdots \cup C_{2s+1}$. Since $G-[2s+1]$ is a forest, every cycle in $G$ has a vertex in $[2s+1]$. Moreover, since distinct $i,j\in[2s+1]$ have $\dist_G(i,j)=d$, $G$ does not have two cycles at distance more than $d$. Hence $G$ does not have a $d$-packing of $k$ cycles each of length at least $\ell$. To reach a contradiction, we now show that for every set $X\subseteq V(G)$ with $|X|\leq s$, $G-B_G(X,d-1)$ has a cycle of length at least $\ell$. Notice that every vertex $v \in V(G)$ either lies in $V(C_i)$ for some $i\in [2s+1]$, or is an inner vertex of $P_{ij}$ for some $i,j\in [2s+1]$. Then $B_G(v,d-1)$ intersects either one or two cycles in $\set{C_1, \dots, C_{2s+1}}$ since each $P_{ij}$ has length $d$. It follows that $B_G(X, d-1)=\bigcup_{v\in X}B_G(v, d-1)$ intersects at most $2|X| \leq 2s$ cycles in $\set{C_1, \dots, C_{2s+1}}$, hence missing at least one of them, as required. $\square$} 
Third, for every fixed $\ell$, the bound $f(k,\ell)\in\Oh(\ell k\log k)$ is optimal as a function of $k$ by standard lower bounds for the Erd\H{o}s--P\'{o}sa theorem.

We now explain how \cref{thm:main-in-intro} is obtained, namely how we
generalize \cref{thm:ErdosPosaCyclesFarApart} to cycles of length at least
$\ell$ while attaining optimal bounding functions for fixed $\ell$. At a high level, our proof follows the same general framework as that of \cite{ErdosPosaCyclesFarApart}. However, some key ingredients break down when moving beyond $\ell=3$, and others stand in the way of optimal bounds.

The proof of \cref{thm:ErdosPosaCyclesFarApart} in
\cite{ErdosPosaCyclesFarApart} exploits the case $\ell=3$ in an essential way:
hitting all cycles means leaving a forest behind. The argument in \cite{ErdosPosaCyclesFarApart} is built
around cycles whose surrounding balls induce unicyclic subgraphs, culminating in an application of the Helly property for subtrees of a tree (via the
Gyárfás--Lehel theorem).
% to a bounded number of forests covering the graph.
For general $\ell$, this structure is simply not available. Namely, a graph with no
cycle of length at least $\ell$ need not resemble a forest and may be locally
dense. Our proof instead relies on two coarser ingredients. First, by a
theorem of \citet{Birmele03}, graphs with no cycle of length at least $\ell$
have treewidth less than $\ell-1$. Thus the forests from
\cite{ErdosPosaCyclesFarApart} become graphs of bounded treewidth, and the
Helly-type argument must be carried out not on these graphs themselves but on
the trees underlying their tree-decompositions. Second, since unicyclic
neighborhoods have no analogue for general $\ell$, we introduce the notion of
the \emph{span} of a cycle $C$, which measures how far apart along $C$ the
BFS-projections of the endpoints of nearby edges can be. Cycles of small span
provide exactly the coarse control that unicyclic balls provided for $\ell=3$:
objects near $C$ project onto short intervals of $C$. A final difficulty is
quantitative: a direct implementation of this strategy produces balls of
radius $\Oh(d\ell)$, and keeping the radius at $\Oh(d)$, independent of
$\ell$, requires an additional padding argument that spreads a small hitting
set on the cycles of the packing into a slightly larger but well-distributed
one.

Regarding the bounding function $f$, the proof in
\cite{ErdosPosaCyclesFarApart} yields $f(k)\in\Oh(k^{18}\polylog k)$, largely
because it invokes a theorem of \citet{gyarfas.lehel:helly} on the
approximate Helly property of subgraphs of a tree with a bounded number of
components. Replacing this ingredient with a theorem of \citet{ALON2002249}
(see \cref{thm:alon} in \cref{sec:tools}) already lowers the bound to
$\Oh(k^3\log k)$, but no further: the approach in
\cite{ErdosPosaCyclesFarApart} applies local arguments to every pair of
cycles in a packing of up to $k-1$ cycles, and therefore cannot produce a subquadratic bound. We
instead analyze all cycles in the packing at once, and in a global manner. Altogether this yields \cref{thm:main-in-intro} with
$f(k,\ell)\in\Oh(\ell k\log k)$ and $g(d)\in\Oh(d)$ and, as a special case,
the optimal bound $f(k)\in\Oh(k\log k)$ in \cref{thm:ErdosPosaCyclesFarApart},
where optimality follows from standard lower bounds for the Erdős--Pósa
theorem.

\cref{thm:ErdosPosaCyclesFarApart,thm:main-in-intro} are part of a larger body
of work. The \defin{far-apart Erd\H{o}s--P\'{o}sa property}, as we call it,
asks for either a large packing of objects that are pairwise far apart, or a
small number of small-radius balls that hit all of the objects. Another
example is the following ``coarse Gallai theorem'' of
\citet{distel2026coarsegallaitheorem}: for every graph $G$ and every set $A$
of its vertices, either $G$ contains $k$ paths with endpoints in $A$
that are pairwise at distance at least $d$, or there is a set $X$ of at most
$f(k)$ vertices such that every such path meets $B_G(X,g(k,d))$. The radius of the
balls depends on $k$, in contrast with \cref{thm:main-in-intro} where it is
$\Oh(d)$ only.

\subsection{Coarse Motivations}

From the point of view of graph minor theory, one can restate the Erdős--Pósa theorem as follows: 
There exists a function $f(k)$ 
such that for every positive integer $k$, and for every graph $G$, either $G$ contains $k$ vertex-disjoint subgraphs each containing a $K_3$ minor, or there exists a subset $X$ of vertices of $G$ with $|X|\leq f(k)$ such that $G-X$ has no $K_3$ minor. 
Our far-apart Erd\H{o}s--P\'{o}sa property (from \cref{subsec:far-apart-property}) is inspired by a new line of research around the Erdős--Pósa theorem, where the usual notion of minors is replaced with that of ``fat'' minors. 

This notion comes from a recently developed field of research called {\em coarse graph theory} that takes its roots in the pioneering work of \citet{GP23}, who studied   
graph minors from the point of view of coarse geometry. 
Quoting Georgakopoulos and Papasoglu, the general goal of coarse graph theory is 

\begin{centering}
    \quad \quad {\it ``to view graphs ‘from far away’ to see their large-scale geometry and its implications.''}
\end{centering}

This novel viewpoint lead to a wealth of new and exciting questions and conjectures in graph theory that have attracted the attention of researchers in the last few years. 
This includes a coarse Menger conjecture (disproved recently by~\citet{NSS25}),  
a coarse grid conjecture (disproved recently by~\citet{AD25}), and 
the following coarse Erd\H{o}s--P\'osa conjecture.\footnote{Very recently, a proof of the conjecture by Sandra Albrechtsen, Marthe Bonamy, Romain Bourneuf, and James Davies was announced at the Focused Workshop on Erd\H{o}s--P\'osa problems held in March 2026 in Będlewo,  Poland.} 
(Necessary definitions are given below.)   

\begin{conj}[Coarse Erd\H{o}s--P\'osa conjecture \cite{GP23}]
\label{conj:coarse-EP}
There exists functions $f:\N \to \N$ and $g:\N \to \N$, with $g(d)\in \Oh(d)$, such that, for all positive integers $k$ and $d$, and for every graph $G$, either $G$ contains $k$ minor-models of $K_3$ that are $d$-fat and pairwise at distance more than $d$, or there exists a subset $X$ of vertices of $G$ with $|X|\leq f(k)$ such that  $G-B_G(X,g(d))$ has no $d$-fat $K_3$ minor-model. 
\end{conj}

To explain the notion of fat minor-models, it will be helpful to first consider the following way of defining the usual notion of minor-models: A \defin{minor-model} of a graph $H$ in a graph $G$ is a collection of vertex-disjoint connected subgraphs $\set{M_v:v\in V(H)}$ and a collection of internally-disjoint paths $\set{P_e: e\in E(H)}$ such that for every edge $e=uv$ of $H$, the path $P_e$ has one endpoint in $M_u$ and the other endpoint in $M_v$, and none of its internal vertices is in any subgraph $M_w$ with $w\in V(H)$. 
Now, given a positive integer $d$, a minor-model of $H$ in $G$ is \defin{$d$-fat} if it satisfies the following extra properties:
% \defin{$d$-fat minor-model} of $H$ in $G$ is a minor-model of $H$ in $G$ as above satisfying the following extra properties: 
\begin{itemize}
\item $M_u$ and $M_v$ are at distance at least $d$ in $G$ for every two distinct vertices $u,v\in V(H)$;
\item $P_e$ and $P_f$ are at distance at least $d$ in $G$ for every two distinct edges $e,f\in E(H)$, and 
\item $M_u$ and $P_e$ are at distance at least $d$ in $G$ for every vertex $u\in V(H)$ and edge $e\in E(H)$ not incident to $u$. 
\end{itemize}
Informally, every two objects from the definition of minor-model are required to be far apart in $G$, except when they are required to intersect as per the original definition of minor. It might be helpful to think of fat minors as minors that would survive even if some disjoint balls of bounded radius were contracted into vertices. Alternatively, at a very intuitive level, the difference between minors and fat minors is that minors are essentially subgraphs that are allowed to be spread out, whereas fat minors are required to be spread out for the purpose of being ``seen'' even when looking at the graph from far away.

%It is important to distinguish~\cref{thm:main-in-intro} and~\cref{conj:coarse-EP} as separate problems since neither immediately implies the other. This is because all pairs of vertices in a long cycle are allowed to be close (or even adjacent) to each other in $G$, but there are many pairs of vertices in a $d$-fat $K_3$ minor-model that are required to have distance at least $d$ in $G$. On the other hand, a $d$-fat minor-model of $K_3$ must contain a cycle of length at least $6d$.
%Thus,~\cref{conj:coarse-EP}, if true, suggests that a statement similar to~\cref{thm:ErdosPosaCyclesFarApart} could hold for long cycles instead of all cycles, for instance, a statement like~\cref{thm:main-in-intro}.

\cref{thm:main-in-intro} and \cref{conj:coarse-EP} are closely related but
incomparable statements: neither implies the other. A first difference lies in
how the parameters interact. In \cref{conj:coarse-EP}, a single parameter $d$
governs everything: the fatness of the minor-models, the distances between
them, and, as a consequence, the lengths of the cycles involved, since every
$d$-fat minor-model of $K_3$ contains a cycle of length at least $6d$. In
\cref{thm:main-in-intro}, by contrast, the length threshold $\ell$ and the
distance parameter $d$ are independent of each other. A second difference lies
in the nature of the objects being packed. A $d$-fat $K_3$ minor-model is
intrinsically spread out: many pairs of its vertices are required to be at
distance at least $d$ in $G$. A cycle of length at least $\ell$ carries no such
requirement: any two of its vertices may be close, or even adjacent, in $G$.
For this reason, \cref{conj:coarse-EP} does not imply \cref{thm:main-in-intro},
even in the regime $\ell\leq 6d$: a hitting set for $d$-fat $K_3$ minor-models
need not come close to any long cycle. (For instance, large complete graphs
contain no $2$-fat $K_3$ minor-model but many long cycles.) Conversely,
\cref{thm:main-in-intro} does not imply \cref{conj:coarse-EP}, since a packing
of pairwise far-apart long cycles need not yield even a single fat $K_3$ minor-model.

This paper is organised as follows. 
In \cref{sec:tools}, we introduce the main tools we use in our proofs, and in \cref{sec:proof}, we prove \cref{thm:main-in-intro}.

\section{Tools}
\label{sec:tools}

Our proof makes use of three tools. The first is a key ingredient of Simonovits's~\cite{Simonovits67} proof of the Erd\H{o}s--P\'{o}sa theorem, which was published in 1967 just two years after the original paper. The second is a beautiful statement generalizing the Helly property of a family of subtrees in a fixed host tree, which is due to \citet{ALON2002249} (improving on an earlier result of Gyárfás and Lehel~\cite{gyarfas.lehel:helly}). The last is a theorem of \citet{Birmele03} which says that graphs with no long cycles have bounded treewidth.

For all positive integers $k$, define\footnote{Here and throughout, $\log x$ denotes the base-$2$ logarithm of $x$.}
\[
  \mathdefin{s(k)}\coloneqq 
    \begin{cases}
      \lceil 4k(\log k + \log\log k +4)\rceil & \textrm{if $k\geq2$,}\\
      2 & \textrm{if $k=1$.}
    \end{cases}
\]

\begin{thm}[\citet{Simonovits67}]\label{thm:simonovits}
  Let $k$ be a positive integer and let $G$ be a graph with all vertices of degree $2$ or $3$. If $G$ contains at least $s(k)$ vertices of degree $3$, then $G$ contains $k$ pairwise vertex-disjoint cycles.
\end{thm}

If a graph has at least one vertex, then it is \defin{non-null}.

\begin{thm}[\citet{ALON2002249}]\label{thm:alon}
For all positive integers $k$ and $c$, for every tree $T$ and every collection $\mathcal{A}$ of non-null subgraphs of $T$ each having at most $c$ components, either
\begin{enumeratealph}
    \item $\mathcal{A}$ has $k$ pairwise vertex-disjoint members, or\label{item:packing:alon}
    \item there exists a subset $X\subseteq V(T)$ with $|X|\leq 2 c^2(k-1)$ such that $X\cap V(A)\not=\emptyset$ for all $A\in \mathcal{A}$.\label{item:covering:alon}
\end{enumeratealph}
\end{thm}

\begin{comment}
% 
% Alon's theorem for tree-decompositions.
% 

\begin{thm}[\citet{ALON2002249}]\label{thm:alon}
For all integers $k,c,t\geqslant 1$, for every graph $F$ of treewidth less than $t$, for every collection $\mathcal{A}$ of non-null subgraphs of $F$ each having at most $c$ components, either
\begin{enumeratealph}
    \item $\mathcal{A}$ has $k$ pairwise vertex-disjoint members, or\label{item:packing:alon}
    \item there exists a subset $X\subseteq V(F)$ with $|X|\leq 2 t c^2(k-1)$ such that $X\cap V(A)\not=\emptyset$ for all $A\in \mathcal{A}$.\label{item:covering:alon}
\end{enumeratealph}
\end{thm}
\end{comment}

% Finally, we need the fact that graphs with no long cycles have bounded treewidth.

\begin{thm}[Birmele~\cite{Birmele03}]
\label{thm:birmele}
For every integer $\ell\geq3$, 
every graph containing no cycle of length at least $\ell$ has treewidth less than $\ell-1$.
\end{thm}

\section{Proof of the Main Result}
\label{sec:proof}

We denote by \defin{$\NN$} the set of nonnegative integers.
For positive integers $k$, we write \defin{$[k]$} as a
compact form of $\set{1,\dots,k}$. For sets $S$, let \defin{$\binom{S}{2}$} be the set of all subsets of $S$ that have size $2$.

We consider simple, finite, and undirected graphs.
Formally, a graph is pair $(V,E)$ consisting of a finite set of \defin{vertices} $V$ with $V\cap \binom{V}{2}=\emptyset$, and a set of \defin{edges} $E\subseteq \binom{V}{2}$.

Let $G$ be a graph. 
We use the notations $\mathdefin{V(G)}$ and $\mathdefin{E(G)}$ to denote the vertex set of $G$ and the edge set of $G$, respectively. A graph $H$ is a \defin{subgraph} of $G$, written as \defin{$H\subseteq G$}, if $V(H)\subseteq V(G)$ and $E(H)\subseteq E(G)$.
A subgraph of $G$ is \defin{spanning} if its vertex set equals $V(G)$. For a set $S\subseteq V(G)$, $\mathdefin{G[S]}$ denotes the subgraph of $G$ with vertex set $S$ and edge set $\binom{S}{2}\cap E(G)$, and $\mathdefin{G-S}\coloneqq G[V(G)\setminus S]$.
% For a vertex set $S$, $\mathdefin{G[S]}$ denotes the subgraph of $G$ induced by the vertices in $S\cap V(G)$, and $\mathdefin{G-S}\coloneqq G[V(G)\setminus S]$.
For a set $Z\subseteq \binom{V(G)}{2}$, $\mathdefin{G\cup Z}$ is the graph with vertex set $V(G)$ and edge set $E(G)\cup Z$; and $\mathdefin{G-Z}$ is the graph with vertex set $V(G)$ and edge set $E(G)\setminus Z$.
% For a set $Z$ of edges of $G$, $\mathdefin{G-Z}$ denotes the subgraph of $G$ obtained by removing the edges in $Z$ from $G$. For a set $Z\subseteq \binom{V(G)}{2}$, $\mathdefin{G\cup Z}$ is the graph with vertex set $V(G)$ and edge set $E(G)\cup Z$.
For two subsets $A$ and $B$ of $V(G)$, we say that an edge $uv$ of $G$ with $u\in A$ and $v\in B$ is \defin{between} $A$ and $B$. For subgraphs $H_1$ and $H_2$ of $G$, $\mathdefin{H_1\cup H_2}$ is the graph with vertex set $V(H_1)\cup V(H_2)$ and edge set $E(H_1)\cup E(H_2)$. For a collection $\mathcal{H}$ of subgraphs of $G$, \defin{$\bigcup\mathcal{H}$} is the graph with vertex set $\bigcup_{H\in \mathcal{H}} V(H)$ and edge set $\bigcup_{H\in \mathcal{H}} E(H)$.

%Let $k$ be a nonnegative integer.
A \defin{path} is a graph 
with the vertex set $\set{v_0,\ldots,v_k}$ and the 
edges $\set{v_{i-1}v_{i} : i \in [k]}$, where $k\geq 0$ and $v_0,\dots,v_k$ are pairwise distinct. 
We often refer to a path by a natural sequence of its vertices, 
writing, say, $v_0\cdots v_k$ or equivalently $v_k \cdots v_0$.
Writing a path as $v_0\cdots v_k$ fixes an underlying orientation of the path, where $v_0$ is the first vertex and $v_k$ is the last vertex.
We say that $v_0$ and $v_k$ are the \defin{endpoints} of $v_0 \cdots v_k$.
We also say that $v_0 \cdots v_k$ \defin{starts} at $v_0$ and \defin{ends} at $v_k$. 
Given a path $P$ and two vertices $v, w$ of $P$, we denote by $\mathdefin{vPw}$ the subpath of $P$ from $v$ to $w$. 

Suppose $A$ and $B$ are vertices or sets of vertices in a graph $G$. 
A path $P$ with $P\subseteq G$ is an \defin{$(A,B)$-path} in $G$ if $P$ starts at (a vertex in) $A$ and ends at (a vertex in) $B$, and the inner vertices of $P$ are disjoint from $A\cup B$. 

A \defin{cycle} is a graph with at least three vertices such that deleting any one edge results in a path. 
A \defin{walk} in a graph $G$ is a sequence $v_0 v_1 \cdots v_k$ of vertices of $G$, with $k\geq 0$, such that $v_{i-1}v_{i} \in E(G)$ for every $i\in [k]$. 
Note that $v_0,\dots,v_k$ are not necessarily distinct.
Often we speak of a walk $v_0 v_1 \cdots v_k$ as a graph, in such cases we are referring to the graph with vertex set $\set{v_0, \dots, v_k}$ and edge set $\set{v_{i-1}v_i:i\in [k]}$.

For walks, paths or edges $P=v_0 v_1\cdots v_s$ and $Q=v_s v_{s+1}\cdots v_t$, with $0\leq s\leq t$, let $\mathdefin{P\cup Q}$ be the walk (or path if it is one) whose underlying sequence of vertices is $v_0 v_1\cdots v_t$.

A \defin{forest} is a graph with no cycles.
A \defin{tree} is a connected forest. 

A \defin{tree-decomposition} of a graph $G$ is a function $\beta$ from the vertex set of some tree $T$ to the power set of $V(G)$, such that for every vertex $v\in V(G)$, the set $\set{t\in V(T) : v\in \beta(t)}$ induces a non-null tree in $T$; and for every edge $uv\in E(G)$, there exists $t\in V(T)$ with $\set{u,v}\subseteq\beta(t)$. The \defin{width} of $\beta$ is $\max_{t\in V(T)}|\beta(t)|-1$. The \defin{treewidth} of $G$, denoted by \defin{$\tw(G)$}, is the minimum width of a tree-decomposition of $G$.

The length, $\mathdefin{\len(P)}$, of a path $P$  
is the number of edges in $P$. 
Similarly, the length of a cycle $C$ is the number of edges in $C$ and is denoted $\mathdefin{\len(C)}$. 
The \defin{distance} between two vertices $x$ and $y$ of $G$, denoted by $\mathdefin{\dist_G(x,y)}$, is 
minimum length of an $(x,y)$-path in $G$ if there is one, 
or \defin{$\infty$} if there is none. For non-empty subsets $X$ and $Y$ of $V(G)$, define $\mathdefin{\dist_G(X,Y)}\coloneqq \min\set{\dist_G(x,y):(x,y)\in X\times Y}$. For an integer $r\ge 0$ and a vertex $x$ of $G$, the \defin{ball of radius $r$ around $x$} is $\mathdefin{B_G(x,r)}\coloneqq \set{y\in V(G):\dist_G(x,y)\le r}$.  For subsets $X$ of $V(G)$, let $\mathdefin{B_G(X,r)}\coloneqq \bigcup_{x\in X}B_G(x,r)$.
% Nice! B_G(\emptyset, r)=\emptyset.

Let $G$ be a graph and let $d$ be a nonnegative integer. A collection $\mathscr{C}$ of subgraphs of $G$ is a \defin{$d$-packing} if $\dist_G(V(A),V(B))> d$ for every pair of distinct $A,B\in\mathscr{C}$.

Let $G$ be a graph and let $H$ be a non-null subgraph of $G$ that contains at least one vertex from each connected component of $G$. An \defin{$H$-supported BFS-spanning subgraph} of $G$ is an edge-minimal spanning subgraph $U$ of $G$ such that $\dist_U(V(H), v)=\dist_G(V(H), v)$ for every vertex $v\in V(G)$. Observe that $U$ is a spanning forest in $G$ and each component of $U$ contains exactly one vertex of $H$, which we call the \defin{root} of the component. For each $v\in V(U)$ the \defin{$U$-projection} of $v$, denoted by \defin{$\proj_U(v)$}, is the root of the component of $U$ that contains $v$; and \defin{$U\angle{v}$} is the unique path in $U$ between $v$ and $\proj_U(v)$. The \defin{$(H, U)$-span} of an edge $vw$ in $G$ and the \defin{$(H,U)$-span} of $G$ are defined as follows:
\[\mathdefin{\spn(G, H, U, vw)} \coloneqq  \dist_{H}(\proj_U(v), \proj_U(w)),\]
\[\mathdefin{\spn(G, H, U)} \coloneqq  \max(\set{\spn(G, H, U, vw) : vw\in E(G)} \setminus \set{\infty}).\]
The \defin{$H$-span} of $G$ is the number $\max_U \spn(G,H,U)$, where the maximum is taken over all $H$-supported BFS-spanning subgraphs of $G$.

\begin{lem}\label{lem:medium-cycle-or-small-span}
Let $d\geqslant 1$ and $\ell\geqslant 3$ be integers. Let $G$ be a graph and let $C$ be a cycle in $G$ of length at least $\ell$. There exists a cycle $C'$ in $G$ of length at least $\ell$ such that either
\begin{enumeratealph}
    \item $V(C') \subseteq B_G(V(C), 3d)$ and the length of $C'$ is at most $6d+2$, or\label{item:medium:medium-cycle-or-small-span}
    \item $V(C')\subseteq B_G(V(C), 2d)$ and the $C'$-span of $G[B_G(V(C'), d)]$ is at most $\ell$.\label{item:small-span:medium-cycle-or-small-span}
\end{enumeratealph}
\end{lem}

\begin{proof}
We say that a pair $(Q,P)$ is \defin{good} if $Q$ is a non-null path in $C$, $P$ is a path in $G$ of length at most $4d+1$, and $Q\cup P$ is a cycle of length at least $\ell$. Observe that, for any edge $e$ of $C$, $(C-e,e)$ is a good pair. Among all good pairs, let $(Q,P)$ be one that minimises $\len(Q)$. 

Let $D=Q\cup P$ and let $G_D\coloneqq G[B_G(V(D),d)]$.  
Since $Q\subseteq C$ and $\len(P)\leq 4d+1$, we have that 
$V(D)\subseteq B_G(V(C),2d)$. 
Thus, if $\spn(G_D,D,U)\le\ell$ for every $D$-supported BFS-spanning subgraph $U$ of $G_D$, then \ref{item:small-span:medium-cycle-or-small-span} holds for $C'=D$ and there is nothing to prove.

Therefore, we assume there is a $D$-supported BFS-spanning subgraph $U$ of $G_D$ and an edge $vw$ of $G_D$ with $\infty > \spn(G_D,D,U,vw)>\ell$. Let $R\coloneqq U\angle{v}\cup vw\cup U\angle{w}$. Since $\spn(G_D, D, U, vw)>\ell\geq 3$, $vw\not\in E(D)$ and $V(U\angle{v})\cap V(U\angle{w})=\emptyset$, thus $R$ is a path whose edges and inner vertices do not appear in $D$. Also observe that $\len(R) = \len(U\angle{v})+1+\len(U\angle{w})\leq 2d+1$. Now let $x$ and $y$ be the endpoints of $R$. The argument splits into cases depending on the locations of $x$ and $y$ in $D=Q\cup P$.

If $x\in V(P)$ and $y \in V(P)$, then consider the unique cycle $C'$ in $R\cup P$. 
Since $\dist_{D}(x,y)>\ell$, we know that 
$\len(C')\geq \ell$.
On the other hand, $\len(C') \leq \len(R)+\len(P) \leq 2d+1 + 4d+1 = 6d+2$. 
Finally $V(R) \subseteq B_G(V(P),d)$ and $V(P) \subseteq B_G(V(C),2d)$ which implies $V(C') \subseteq B_G(V(C),3d)$. 
This proves that \ref{item:medium:medium-cycle-or-small-span} holds.

Otherwise, we may assume $x\in V(Q)\setminus V(P)$ without loss of generality. 
Then, the graph $D\cup R$ has two cycles that contain $R$. 
One of these cycles, $D'$, contains at most $\lfloor\tfrac{1}{2}\len(P)\rfloor\le 2d$ edges of $P$. 
Note also that $\len(D')\geq \dist_{D}(x,y)>\ell$.  
Let $Q'\coloneqq D'\cap Q$ and $P'\coloneqq D'\cap (P\cup R)$. 
Note that $x$ lies in $Q'$ so $Q'$ is non-null. 
Also, observe that $x$ has two neighbors in $Q$ (since $x \notin V(P)$) and  $D'$ contains only one of these two. 
Hence, $\len(Q')<\len(Q)$. 
Finally, 
$\len(P')= \len(P\cap D')+\len(R)\le 2d+2d+1=4d+1$. 
Altogether, $(Q',P')$ is a good pair and $\len(Q')<\len(Q)$, which contradicts the choice of $(Q,P)$.

This concludes the proof of the lemma.
\end{proof}

\begin{lem}
\label{lem:pre-def-of-pad}
For every cycle $C$, for every $y\in V(C)$, and for every pair of positive integers $(\alpha,\beta)$, there exists a set $Z \subseteq V(C)$ such that
\begin{enumerate}
\item $y\in Z$;\label{item:y-in-Z:pre-def-of-pad}
\item $|Z|\leq 2\beta+1$;\label{item:size:pre-def-of-pad}
\item $B_C(Z,\floor{\alpha/2})$ is connected and has size at least $\min \set{\len(C),2(\alpha\beta+\floor{\alpha/2})+1}$;\label{item:connected:pre-def-of-pad}
\item For each vertex $v$ in $C$, if $\dist_C(v,y)\leq \alpha\beta$ then $\dist_C(v,Z)\leq \alpha/2$. \label{item:padding-property:pre-def-of-pad}
\end{enumerate}
\end{lem}

\begin{proof}
Fix an orientation of $C$. Let $v_0 v_1 \cdots  v_{\alpha\beta}$ be the walk in $C$ such that $v_0=y$, and $v_{i+1}$ is the successor of $v_i$ in the clockwise cyclic ordering of $C$ for all $i$. Similarly, let $u_0 u_1 \cdots u_{\alpha\beta}$ be the walk in $C$ such that $u_0=y$, and $u_{i+1}$ is the successor of $u_i$ in the anticlockwise cyclic ordering of $C$ for all $i$. Let $Z\coloneqq \set{v_\alpha, v_{2\alpha}, \dots, v_{\alpha\beta}}\cup \set{y}\cup \set{u_\alpha, u_{2\alpha}, \dots, u_{\alpha\beta}}$. We show that $Z$ satisfies the outcome of the lemma. Observe that \ref{item:y-in-Z:pre-def-of-pad} and \ref{item:size:pre-def-of-pad} are immediate. For any pair of consecutive vertices in the sequence $v_{\beta\alpha}, v_{(\beta-1)\alpha}, \dots, v_{\alpha}, y, u_{\alpha}, u_{2\alpha}, \dots, u_{\beta\alpha}$, the union of their radius $\lfloor \alpha/2 \rfloor$ balls is connected, thus $B_C(Z,\floor{\alpha/2})$ is connected. If $|B_C(Z,\floor{\alpha/2})|<\len(C)$, then $v_{\alpha\beta+\floor{\alpha/2}} \cdots v_2 v_1 y u_1 u_2 \cdots u_{\alpha\beta+\floor{\alpha/2}}$ defines a path in $C$, which implies $|B_C(Z,\floor{\alpha/2})| \geq 2(\alpha\beta+\floor{\alpha/2})+1$. Hence \ref{item:connected:pre-def-of-pad} holds. Finally, \ref{item:padding-property:pre-def-of-pad} follows from the fact that every vertex in $\set{v_1, v_2,\dots, v_{\alpha\beta}}\cup \set{y} \cup \set{u_1, u_2,\dots, u_{\alpha\beta}}$ has distance at most $\alpha/2$ from some vertex in $Z$.
\end{proof}

For cycles $C$, vertices $y\in V(C)$, and positive integers $\alpha,\beta$, let $\mathdefin{\pad(C,y,\alpha,\beta)}$ denote the corresponding set $Z$ constructed in \cref{lem:pre-def-of-pad}.

Let $H$ be a disjoint union of cycles $C_1\cup \cdots \cup C_p$, let $Y\subseteq V(H)$, and let $(\alpha,\beta)$ be a pair of positive integers. Define
\[\mathdefin{\pad(H,Y,\alpha,\beta)} \coloneqq  \bigcup_{i\in [p]}\;\; \bigcup_{y\in Y\cap V(C_i)} \pad(C_i,y,\alpha,\beta).\]

The outcomes of \cref{lem:pre-def-of-pad} generalise to $\pad(H,Y,\alpha,\beta)$ in the following way:

\begin{cor}\label{cor:def-of-pad}
If $H$ is a disjoint union of cycles, $Y\subseteq V(H)$, and $(\alpha,\beta)$ is a pair of positive integers, then
\begin{enumerate}
\item $Y\subseteq \pad(H,Y,\alpha,\beta)\subseteq V(H)$;\label{item:Y-in-pad:def-of-pad}
\item $|\pad(H,Y,\alpha,\beta)|\leq |Y|(2\beta+1)$;\label{item:size:def-of-pad}
\item for each component $C$ of $H$, every component of $C[V(C)\cap B_H(\pad(H,Y,\alpha,\beta),\lfloor\alpha/2\rfloor)]$ has at least $\min \set{\len(C),2(\alpha\beta+\lfloor\alpha/2\rfloor)+1}$ vertices;\label{item:connected:def-of-pad}
\item  For each vertex $v$ in $H$, if $\dist_H(v,Y)\leq \alpha\beta$ then $\dist_H(v,\pad(H,Y,\alpha,\beta))\leq \alpha/2$. \label{item:padding-property:def-of-pad}
\end{enumerate}
\end{cor}

\begin{proof}
\ref{item:Y-in-pad:def-of-pad} and \ref{item:size:def-of-pad} follow immediately from \cref{lem:pre-def-of-pad}. Observe that for each component $C$ of $H$, $V(C)\cap B_H(\pad(H,Y,\alpha,\beta), \lfloor\alpha/2\rfloor)=\bigcup_{y\in Y\cap V(C)}B_C(\pad(C, y, \alpha,\beta), \lfloor\alpha/2\rfloor)$, hence \ref{item:connected:def-of-pad} holds by \cref{lem:pre-def-of-pad}. To see that \ref{item:padding-property:def-of-pad} holds, consider any $v\in V(H)$ with $\dist_H(v, Y)\leq \alpha\beta$. Let $C$ be the component of $H$ that contains $v$, then there exists $y\in Y\cap V(C)$ such that $\dist_C(v,y)\leq \alpha\beta$. Then \cref{lem:pre-def-of-pad} implies that $\dist_C(v,\pad(C,y,\alpha,\beta))\leq \alpha/2$, therefore $\dist_H(v,\pad(H,Y,\alpha,\beta))\leq \alpha/2$. 
\end{proof}

\begin{lem}\label{lem:big-span-and-intersecting-balls}
Let $k,d,r,R,\ell$ be positive integers with $R\geq r > d$ and $\ell\geqslant 3$. Let $G$ be a graph, let $\mathscr{C}$ be a non-empty $2d$-packing of cycles in $G$ each of length at least $\ell$, let $U$ be a $\bigcup\mathscr{C}$-supported BFS-spanning subgraph of $G_R\coloneqq G[B_G(V(\bigcup\mathscr{C}), R)]$, and let $W$ be the set of endpoints of all edges $e\in E(G_R)$ with $\spn(G_R,\bigcup\mathscr{C}, U, e)>\ell$. If the $\bigcup\mathscr{C}$-span of $G_d\coloneqq G[B_G(V(\bigcup\mathscr{C}), d)]$ is at most $\ell$, then either 
\begin{enumeratealph}
    \item $G$ contains a $d$-packing of $k$ cycles each of length at least $\ell$, or\label{item:packing:big-span-and-intersecting-balls}
    \item \label{item:alternative:big-span-and-intersecting-balls} there exists $X\subseteq V(\bigcup\mathscr{C})$ with $|X|\leq (s(k)-1)(2\ceil{\ell/2}+1)$ such that
    \[W\cup \bigcup_{\set{C,D}\in \binom{\mathscr{C}}{2}} \left(B_G(V(C), r)\cap B_G(V(D), r) \right)\subseteq B_G(W, r-d-1) \subseteq B_G(X, R+r+d).\]
\end{enumeratealph}
\end{lem}

\begin{proof}
Let $\mathdefin{\mathcal{E}}\coloneqq \set{e\in E(G_R) : \spn(G_R, \bigcup\mathscr{C}, U, e)>\ell}$. 
(Thus, $W$ is the set of endpoints of edges in $\mathcal{E}$.)  
For each $e=uv\in \mathcal{E}$, let
\[\mathdefin{P_e}\coloneqq U\angle{u}\cup uv \cup U\angle{v}.\]
Since $\spn(G_R, \bigcup\mathscr{C}, U, uv)>\ell\geq 3$, $uv\not\in E(\bigcup\mathscr{C})$ and $V(U\angle{u})\cap V(U\angle{v})=\emptyset$, thus $P_e$ is a path whose edges and inner vertices do not appear in $\bigcup\mathscr{C}$.

For each edge $e=uv\in \mathcal{E}$, let $\mathdefin{\proj_U(e)} \coloneqq  \set{\proj_U(u), \proj_U(v)}$. Let \defin{$H$} be the auxiliary graph with vertex set $\mathcal{E}$, where two distinct elements $e,f\in V(H)$ are adjacent if and only if
\[\dist_G(V(P_e), V(P_f)) \leq d \quad \text{or}\quad \dist_{\bigcup\mathscr{C}}(\proj_U(e), \proj_U(f)) \leq d\ell.\]
Let \defin{$I$} be a maximal independent set in $H$ and consider the graph \[
\textstyle\mathdefin{G'}\coloneqq \bigcup \mathscr{C} \cup \bigcup_{e\in I}P_e.
\]
Recall that $\mathscr{C}$ is a collection of pairwise vertex-disjoint cycles.
Also since $\dist_G(V(P_e), V(P_f))>d$ for all distinct $e,f\in I$, $(P_e : e\in I)$ is a collection pairwise vertex-disjoint paths. Furthermore, no edge or inner vertex of $P_e$ appears in $\bigcup\mathscr{C}$ for all $e\in \mathcal{E}$. Therefore, every vertex $v\in V(G')$ satisfies $\deg_{G'}(v) \in \{2,3\}$, and the degree-$3$ vertices of $G'$ are precisely the endpoints of paths in $(P_e : e\in I)$. Hence $G'$ has exactly $2|I|$ degree-$3$ vertices.

The argument splits into the following two cases: \textbf{(1)} $2|I|\geqslant s(k)$, in which case it is shown that \ref{item:packing:big-span-and-intersecting-balls} holds, and \textbf{(2)} $2|I|\leq s(k)-1$, in which case it is shown that \ref{item:alternative:big-span-and-intersecting-balls} holds.

\textbf{Case (1)} $2|I|\geqslant s(k)$: A \defin{$\mathscr{C}$-segment} 
is a path $S$ in $\bigcup\mathscr{C}$ between an endpoint of $P_e$ and an endpoint of $P_f$ for some $e,f\in I$, such that the inner vertices of $S$ are degree-$2$.

\begin{clm}\label{clm:cycles-in-G'-are-long}
Every cycle in $G'$ has length at least $\ell$.
\end{clm}

\begin{clmproof}
Every cycle in $\mathscr{C}$ has length at least $\ell$ and 
every cycle in $G'$ that is not in $\mathscr{C}$ contains a $\mathscr{C}$-segment. 
Thus, it suffices to show that every $\mathscr{C}$-segment has length greater than $\ell$.
Let $S$ be a $\mathscr{C}$-segment whose endpoints $x$ and $y$ correspond with an endpoint $x$ of $P_e$ and an endpoint $y$ of $P_f$ for some $e,f\in I$. If $e=f$, then $\len(S) \geqslant \spn(G_R, \bigcup\mathscr{C}, U, e) > \ell$ since $e\in \mathcal{E}$. If $e\not=f$, then $\len(S) \geqslant \dist_{\bigcup\mathscr{C}}(x,y) \geqslant \dist_{\bigcup\mathscr{C}}(\proj_U(e), \proj_U(f)) > d\ell \geqslant \ell$.
\end{clmproof}

\begin{clm}\label{clm:disjoint-cycles-in-G'-are-far-apart}
For every pair of vertex-disjoint cycles $D$ and $D'$ in $G'$, $\dist_G(V(D), V(D')) > d$.
\end{clm}

\begin{clmproof}
Assume for a contradiction that there exists a pair of vertex-disjoint cycles $D$ and $D'$ in $G'$ such that $\dist_G(V(D), V(D'))\leq d$. Consider any $u\in V(D)$ and $u'\in V(D')$ with $\dist_G(u,u') \leq d$.

If $u\in V(P_e)$ and $u'\in V(P_f)$ for some $e,f\in I$ with $P_e\subseteq D$ and $P_f\subseteq D'$, then since $D$ and $D'$ are vertex-disjoint, $e\not=f$. Therefore $\dist_G(u,u') \geq \dist_G(V(P_e), V(P_f)) > d$, a contradiction. Hence it may be assumed that $u\not\in V(P_e)$ for every $e\in I$ with $P_e\subseteq D$. Consequently, there exists a unique $C\in \mathscr{C}$ such that $u\in V(C)$.

Let $P$ be a shortest $(u,u')$-path in $G$. Since $\dist_G(u,u')\leq d$, $\len(P)\leq d$, so $V(P)\subseteq B_G(V(C), d)$. In particular, since $\mathscr{C}$ is a $2d$-packing in $G$, $u'\in B_G(V(C), d)$ and $u'\not\in B_G(V(\bigcup\mathscr{C})\setminus V(C), d)$. Therefore $\proj_U(u')\in V(C)$, $V(U\angle{u'})\subseteq B_G(V(C), d)$, and $\len(U\angle{u'})\leq d$. Furthermore, either $u'\in V(C)$ or $u'$ is an internal vertex of $P_e$ for some $e\in I$ with $P_e\subseteq D'$. In the former case, $\proj_U(u')=u'\in V(D')$, and in the latter case $\proj_U(u')\in V(U\angle{u'})\subseteq V(P_e) \subseteq V(D')$. Hence $\proj_U(u')\in V(D')$ in both cases.

Let $v_0, v_1, \dots, v_s$ be the sequence of vertices along $P$ such that $v_0=u$, $v_s=u'$, and $s\leq d$. Since $V(P)\subseteq B_G(V(C), d) \subseteq V(G_d)$, an assumption of the lemma implies that $\spn(G_R, \bigcup\mathscr{C}, U, e) = \spn(G_d, \bigcup\mathscr{C}, U[V(G_d)], e) \leq \ell$ for every $e\in E(P)$. Therefore:
\begin{align}
    \dist_C(u,\proj_U(u'))=\dist_{\bigcup\mathscr{C}}(u,\proj_U(u')) &\leq \sum_{i=1}^{s}\dist_{\bigcup\mathscr{C}}(\proj_U(v_{i-1}), \proj_U(v_i))\notag\\
    &= \sum_{i=1}^{s}\textstyle\spn(G_R, \bigcup\mathscr{C}, U, v_{i-1}v_i)\leq s\ell \leq d\ell.\tag{$\mathbf{*}$}\label{eq:dist-between-u-and-proj_U(u')}
\end{align}
Now consider a shortest $(u,\proj_U(u'))$-path $S$ in $C$. Recall that $u\in V(D)$ and $\proj_U(u')\in V(D')$. Let $x$ be the first vertex in $S$ starting from $u$ such that $x\in V(P_e)$ for some $e\in I$ with $P_e\subseteq D$. Let $y$ be the first vertex in $S$ starting from $\proj_U(u')$ such that $y\in V(P_f)$ for some $f\in I$ with $P_f\subseteq D'$. Since $D$ and $D'$ are vertex-disjoint, both vertices $x$ and $y$ exist and $e\not=f$. Note that $\dist_C(x,y)\leq \dist_C(u, \proj_U(u')) \leq d\ell$ by \eqref{eq:dist-between-u-and-proj_U(u')}. On the other hand, $x\in V(C)\cap V(P_e)\subseteq \proj_U(e)$, $y\in V(C)\cap V(P_f) \subseteq \proj_U(f)$, and $e$ and $f$ are distinct elements of $I$ together imply $\dist_C(x,y)\geq \dist_{\bigcup\mathscr{C}}(\proj_U(e), \proj_U(f)) > d\ell$, a contradiction.
\end{clmproof}

Since $2|I|\geqslant s(k)$, $G'$ has at least $s(k)$ degree-$3$ vertices. So \cref{thm:simonovits} implies that $G'$ contains $k$ pairwise vertex-disjoint cycles. By \cref{clm:cycles-in-G'-are-long}, each of these cycles has length at least $\ell$. By \cref{clm:disjoint-cycles-in-G'-are-far-apart}, these cycles form a $d$-packing in $G$. Thus outcome \ref{item:packing:big-span-and-intersecting-balls} holds.

\textbf{Case (2)} $2|I|\leq s(k)-1$: Let
\[\textstyle Y\coloneqq \bigcup_{e\in I}\proj_U(e) \quad \text{and}\quad X\coloneqq \pad(\bigcup\mathscr{C}, Y, 2d, \ceil{\ell/2}).\]
Note that $Y\subseteq V(\bigcup\mathscr{C})$ and $|Y|= 2|I| \leq s(k)-1$. Then by \cref{cor:def-of-pad}, $X\subseteq V(\bigcup\mathscr{C})$ and $|X|\leq |Y|(2\ceil{\ell/2}+1)\leq (s(k)-1)(2\ceil{\ell/2}+1)$.

\begin{clm}\label{clm:W-is-hit}
$W\subseteq B_G(X, R+2d+1)$.
\end{clm}

\begin{clmproof}
Consider any $e\in \mathcal{E}$ and let $v$ be an endpoint of $e$. Note that $e\in V(H)$. If $e\in I$, then $\proj_U(v)\in \proj_U(e) \subseteq Y\subseteq X$, so $v\in B_G(X,R)$ as desired. Hence it may be assumed that $e\in V(H)\setminus I$. Since $I$ is a maximal independent set in $H$, there exists $f\in I$ such that $ef\in E(H)$. Then there are two cases, \textbf{(i)} $\dist_G(V(P_e), V(P_f))\leq d$, or \textbf{(ii)} $\dist_{\bigcup\mathscr{C}}(\proj_U(e), \proj_U(f))\leq d\ell$.

\textbf{Case (i)} $\dist_G(V(P_e), V(P_f))\leq d$: Let $u\in V(P_e)$ and $u'\in V(P_f)$ such that $\dist_G(u,u') = \dist_G(V(P_e), V(P_f))$, and let $P$ be a shortest $(u,u')$-path in $G$. Let $Q$ be the walk in $G$ starting from $\proj_U(u')$ and going to $u'$ along $P_f$, then following $P$ to $u$, and finally going to $v$ along $P_e$. Let $s$ be the length of $\proj_U(u')Qu'$. Since $U$ is a $\bigcup\mathscr{C}$-supported BFS-spanning subgraph of $G_R$, $\dist_G(V(\bigcup\mathscr{C}), u') = s$. Then $s\leq \dist_G(V(\bigcup\mathscr{C}), u)+\dist_G(u,u')$, which implies $\dist_G(V(\bigcup\mathscr{C}), u) \geqslant s-d$. Since each endpoint of $e$ is at distance at most $R$ from $V(\bigcup\mathscr{C})$ in $G$, the length of $u Q v$ is at most $R+1-(s-d)$. (The $+1$ accounts for the fact that $uQv$ may contain the edge $e$.) 
It follows that $Q$ has length at most $s+d+R+1-(s-d) =R+2d+1$. Recall that, since $f\in I$, $\proj_U(u')\in \proj_U(f)\subseteq Y$, thus $v\in B_G(\proj_U(u'), R+2d+1) \subseteq B_G(Y, R+2d+1) \subseteq B_G(X, R+2d+1)$, as desired.

\textbf{Case (ii)} $\dist_{\bigcup\mathscr{C}}(\proj_U(e), \proj_U(f))\leq d\ell$: Since $\proj_U(f) \subseteq Y$, $\dist_{\bigcup\mathscr{C}}(w,Y)\leq d\ell \leq 2d\ceil{\ell/2}$ for some $w\in \proj_U(e)$. Then \cref{cor:def-of-pad} implies that $\dist_{\bigcup\mathscr{C}}(w, X)\leq d$. Hence $v\in B_G(w, R+1) \subseteq B_G(X,R+d+1)$, as desired.
\end{clmproof}

\begin{clm}\label{clm:intersections-are-hit}
$\bigcup_{\set{C,D}\in \binom{\mathscr{C}}{2}} \left(B_G(V(C), r)\cap B_G(V(D), r)\right) \subseteq B_G(W,r-d-1)$.
\end{clm}

\begin{clmproof}
By an assumption of the lemma, $r>d$. Let $p\coloneqq |\mathscr{C}|$ and enumerate $\mathscr{C}$ as $C_1, C_2, \dots, C_p$. For each $i\in [p]$, let $U_i\coloneqq U[B_U(V(C_i), r)]$. Then $(V(U_i) : i\in [p])$ is a partition of $B_G(\bigcup\mathscr{C}, r)$.

Consider each $a\in \bigcup_{\set{C,D}\in \binom{\mathscr{C}}{2}} \left(B_G(V(C), r)\cap B_G(V(D), r)\right)$. Since $a\in B_G(V(\bigcup\mathscr{C}), r)$, 
there exists a unique $i\in [p]$ such that $a\in V(U_i)$. 
Therefore $a\in B_G(V(C_i), r)\cap B_G(V(\bigcup\mathscr{C})\setminus V(C_i), r)$ and $a\not\in V(\bigcup\mathscr{C})\setminus V(C_i)$.
Let $Q$ be a shortest $(a, V(\bigcup\mathscr{C})\setminus V(C_i))$-path in $G$. 
Note that $Q$ has an edge and $V(Q) \subseteq B_G(\bigcup\mathscr{C}, r)$. 
Let $u_a v_a$ be the first edge on $Q$ such that $u_a\in V(U_i)$ and $v_a\not\in V(U_i)$. Thus $v_a\in V(U_j)$ for some $j\in [p]\setminus \set{i}$. Therefore $u_av_a$ is an edge of $G_R$ with $\spn(G_R, \bigcup\mathscr{C}, U, u_av_a)=\infty$, so $\set{u_a,v_a}\subseteq W$. If $u_av_a$ is one of the last $d$ edges of $Q$, then $u_a\in B_G(V(\bigcup\mathscr{C}) \setminus V(C_i), d)$, implying $u_a\not\in V(U_i)$, a contradiction. Hence $\len(aQu_a)\leq \len(Q)-(d+1) \leq r-d-1$ and $a\in B_G(u_a, r-d-1) \subseteq B_G(W, r-d-1)$, as desired.
\end{clmproof}

\cref{clm:W-is-hit} and \cref{clm:intersections-are-hit} together imply
\[W\cup \bigcup_{\set{C,D}\in \binom{\mathscr{C}}{2}}\left(B_G(V(C), r)\cap B_G(V(D), r)\right) \subseteq B_G(W, r-d-1) \subseteq B_G(X, R+r+d),\]
hence \ref{item:alternative:big-span-and-intersecting-balls} holds.

This concludes the proof of the lemma.
\end{proof}

\begin{lem}\label{lem:cycle-helly-for-small-span-collection}
Let $k,r,\ell$ be positive integers with $\ell\geq 3$. Let $G$ be a graph, let $\mathscr{C}$ be a non-empty collection of pairwise vertex-disjoint cycles in $G$,
and let $U$ be a $\bigcup\mathscr{C}$-supported BFS-spanning subgraph of $G_r\coloneqq G[B_G(V(\bigcup\mathscr{C}), r)]$. If $\mathcal{H}$ is a collection of non-null connected subgraphs of $G_r$ such that $\spn(G_r, \bigcup\mathscr{C}, U, e) \leq \ell$ for every $e\in E(\bigcup\mathcal{H})$, then either
\begin{enumerate}
    \item $\mathcal{H}$ contains $k$ pairwise vertex-disjoint members, or\label{item:packing:cycle-helly-for-small-span-collection}
    \item there exists $Y\subseteq V(\bigcup\mathscr{C})$ with $|Y|\leq k-1+|\mathscr{C}|$ such that $B_U(B_{\bigcup\mathscr{C}}(Y,\lfloor \ell/2 \rfloor), r)\cap V(H)\not=\emptyset$ for all $H\in \mathcal{H}$.\label{item:alternative:cycle-helly-for-small-span-collection}
\end{enumerate}
\end{lem}

\begin{proof}
For each $C\in \mathscr{C}$, arbitrarily choose a vertex $\mathdefin{v_C}\in V(C)$ and let $\mathdefin{P_C}\coloneqq C-B_C(v_C, \lfloor \ell/2\rfloor)$. Then $(P_C : C\in \mathscr{C})$ is a collection of pairwise vertex-disjoint paths (note that $V(P_C)=\emptyset$ is possible). Let $\mathdefin{F}\coloneqq \bigcup_{C\in \mathscr{C}}P_C$. For each $H\in \mathcal{H}$, let $\mathdefin{\proj_U(H)}\coloneqq \set{\proj_U(v):v\in V(H)}$.

Consider each $H\in \mathcal{H}$. For each $uv\in E(H)$, since $\spn(G_r, \bigcup\mathscr{C}, U, uv)$ is finite, there exists $C\in \mathscr{C}$ such that $\{\proj_U(u), \proj_U(v)\}\subseteq V(C)$. Moreover, since $H$ is connected and $\mathscr{C}$ is a collection of pairwise vertex-disjoint cycles, there exists a unique $\mathdefin{C(H)}\in \mathscr{C}$ such that $\proj_U(H) \subseteq V(C(H))$ (this is true even when $E(H)=\emptyset$). Let \defin{$\mathcal{H}'$} be the set of all $H\in \mathcal{H}$ such that $\proj_U(H)\subseteq V(P_{C(H)})$, and for each such $H$, let \defin{$S_H$} be the minimal subpath of $P_{C(H)}$ with $\proj_U(H)\subseteq V(S_H)$.

Since $F$ is a disjoint union of paths and $(S_H:H\in \mathcal{H}')$ is a collection of paths in $F$, the Helly property for intervals implies that either \textbf{(1)} $(S_H:H\in \mathcal{H}')$ has $k$ pairwise vertex-disjoint members, or \textbf{(2)} there exists $Z\subseteq V(F)$ with $|Z|\leq k-1$ such that $Z\cap V(S_H)\not=\emptyset$ for all $H\in \mathcal{H}'$.

\textbf{Case (1)} $(S_H:H\in \mathcal{H}')$ has $k$ pairwise vertex-disjoint members: Then there exists $\mathcal{P}\subseteq \mathcal{H}'$ with $|\mathcal{P}|=k$ such that $V(S_A)\cap V(S_B)=\emptyset$ for all distinct $A,B\in \mathcal{P}$. Therefore $\proj_U(A)\cap \proj_U(B)=\emptyset$ for all distinct $A,B\in \mathcal{P}$, which implies that $V(A)\cap V(B)=\emptyset$ for all distinct $A,B\in \mathcal{P}$. Hence $\mathcal{P}$ is a collection of $k$ pairwise vertex-disjoint members of $\mathcal{H}$, so \ref{item:packing:cycle-helly-for-small-span-collection} holds.

\textbf{Case (2)} there exists $Z\subseteq V(F)$ with $|Z|\leq k-1$ such that $Z\cap V(S_H)\not=\emptyset$ for all $H\in \mathcal{H}'$: Let $Y\coloneqq Z\cup \{v_C:C\in \mathscr{C}\}$, then $Y\subseteq V(\bigcup\mathscr{C})$ and $|Y|\leq k-1+|\mathscr{C}|$. We now show that $B_{\bigcup\mathscr{C}}(Y,\lfloor \ell/2 \rfloor)\cap \proj_U(H)\not=\emptyset$ for all $H\in \mathcal{H}$, which implies that $B_U(B_{\bigcup\mathscr{C}}(Y,\lfloor \ell/2 \rfloor), r) \cap V(H)\not=\emptyset$ for all $H\in \mathcal{H}$, then \ref{item:alternative:cycle-helly-for-small-span-collection} holds.
Consider any $H\in \mathcal{H}$ and let $C\coloneqq C(H)$.
It may be assumed that $B_C(v_C,\lfloor \ell/2 \rfloor)\cap \proj_U(H)=\emptyset$ and $Z\cap \proj_U(H)=\emptyset$. Then $\proj_U(H)\subseteq V(P_C)$ and $H\in \mathcal{H}'$. Let $z\in Z\cap V(S_H)$ as promised by the premise of the present case. The minimality of $S_H$ implies its endpoints lie in $\proj_U(H)$, thus $z$ is an inner vertex of $S_H$. Since $H$ is connected, there exists a path $q_0q_1\cdots q_s$ in $H$ such that $\proj_U(q_0)$ and $\proj_U(q_s)$ are the two endpoints of $S_H$. Let $p_i\coloneqq \proj_U(q_i)$ for each $i\in \{0, \dots, s\}$. Then $z\in V(S_H)=V(p_0S_H p_s)$  
and $z\not=p_s$. Let $t$ be minimum such that $z\not\in V(p_tS_H p_s)$. Then $t\in [s]$ and $z\in V(p_{t-1}S_H p_s)$. It follows that $p_tS_H p_s$ is a subpath of $p_{t-1}S_H p_s$, and $z\in V(p_{t-1}S_H p_t)$.
Since $z\not\in \proj_U(H)$, $p_{t-1}, z, p_t$ are distinct vertices in $P_C$. Therefore $C[B_C(v_C, \floor{\ell/2}+1)]$ is a path of length $2(\floor{\ell/2}+1)$, implying that the $(p_{t-1}, p_t)$-path in $C$ that does not contain $z$ has length at least $2(\floor{\ell/2}+1) > \ell$. However by an assumption of the lemma $\dist_C(p_{t-1}, p_t) \leq \ell$, therefore $\len(p_{t-1}S_H p_t) \leq \ell$. Consequently $B_C(z,\lfloor \ell/2 \rfloor) \cap \{p_{t-1}, p_t\} \not=\emptyset$, hence $B_{\bigcup\mathscr{C}}(Y,\lfloor \ell/2 \rfloor)\cap \proj_U(H)\not=\emptyset$, as desired.

This concludes the proof of the lemma.
\end{proof}

\begin{lem}\label{lem:hitting-long-cycles-close-to-packing}
Let $k,d,r,\ell$ be positive integers with $r\geq \ceil{d/2}$ and $\ell\geq 3$. Let $R\coloneqq  r+\ceil{d/2}$, let $G$ be a graph, let $\mathscr{C}$ be a non-empty $d$-packing of cycles in $G$ each of length at least $\ell$, let $U$ be a $\bigcup\mathscr{C}$-supported BFS-spanning subgraph of $G_R\coloneqq G[B_G(V(\bigcup\mathscr{C}), R)]$, and let $W$ be the set of endpoints of all edges $e\in E(G_R)$ such that $\spn(G_R, \bigcup\mathscr{C}, U, e) > \ell$. Then either
\begin{enumeratealph}
    \item $G$ contains a $d$-packing of $k$ cycles each of length at least $\ell$, or\label{item:packing:hitting-long-cycles-close-to-packing}
    \item there exists $Y\subseteq V(\bigcup\mathscr{C})$ with $|Y|\leq 2(k-1)$ such that, by letting $G_r\coloneqq G[B_G(V(\bigcup\mathscr{C}), r)]$ and  $X\coloneqq \pad(\bigcup\mathscr{C}, Y, 2\ceil{d/2}+1, \ceil{\ell/2})$, we have that
    \[G_r-\big(B_G(W, \lfloor d/2\rfloor)\cup B_{\bigcup\mathscr{C}\cup U}(X,R)\big)\]
    has no cycles of length at least $\ell$.\label{item:alternative:hitting-long-cycles-close-to-packing}
\end{enumeratealph}

\end{lem}
\begin{proof}
If $|\mathscr{C}| \geq k$, then \ref{item:packing:hitting-long-cycles-close-to-packing} holds, hence it may be assumed that $|\mathscr{C}|\leq k-1$. Let \defin{$\mathcal{D}$} be the set of all cycles in $G_r - B_G(W, \lfloor d/2 \rfloor)$ of length at least $\ell$. For each $D\in \mathcal{D}$, let \defin{$U_D$} be a $D$-supported BFS-spanning subgraph of $G[B_G(V(D), \lceil d/2 \rceil)]$, and let $\mathdefin{H_D}\coloneqq D\cup U_D$.

\begin{clm}\label{clm:edges-of-H_D-have-big-span}
$\spn(G_R, \bigcup\mathscr{C}, U, e)\leq \ell$ for every $D\in \mathcal{D}$ and every $e\in E(H_D)$.
\end{clm}
\begin{clmproof}
Consider any $D\in \mathcal{D}$ and any edge $e\in E(H_D)$. Without loss of generality, $e=uv$ where $u\in B_G(V(D), \ceil{d/2}-1)$. Since $\ceil{d/2}-1 \leq \floor{d/2}$ and $B_G(V(D), \floor{d/2}) \cap W = \emptyset$, $u\not\in W$. This along with $uv\in E(G_R)$ implies that $\spn(G_R, \bigcup\mathscr{C}, U, uv) \leq \ell$, as desired.
\end{clmproof}

Now $\mathdefin{\mathcal{H}} \coloneqq  \set{H_D : D\in \mathcal{D}}$ is a collection of non-null connected subgraphs of $G_R$ that, by \cref{clm:edges-of-H_D-have-big-span}, satisfies $\spn(G_R, \bigcup\mathscr{C}, U, e)\leq \ell$ for every $e\in E(\bigcup\mathcal{H})$. Hence we may proceed by cases depending on the outcome of \cref{lem:cycle-helly-for-small-span-collection}.

\textbf{Case \ref{item:packing:cycle-helly-for-small-span-collection}} $\mathcal{H}$ contains $k$ pairwise vertex-disjoint members: Then there exists $\mathcal{P}\subseteq \mathcal{D}$ with $|\mathcal{P}|=k$ such that $V(H_A)\cap V(H_B)=\emptyset$ for all distinct $A,B\in \mathcal{P}$. Recall that $V(H_D)=V(U_D)=B_G(V(D), \ceil{d/2})$ for all $D\in \mathcal{D}$, thus $\mathcal{P}$ is a $(2\ceil{d/2})$-packing 
of $k$ cycles in $G$ each of length at least $\ell$, so \ref{item:packing:hitting-long-cycles-close-to-packing} holds.

\textbf{Case \ref{item:alternative:cycle-helly-for-small-span-collection}} there exists $Y\subseteq V(\bigcup\mathscr{C})$ with $|Y|\leq k-1+|\mathscr{C}|$ such that for $Z\coloneqq B_{\bigcup\mathscr{C}}(Y,\floor{\ell/2})$, $B_U(Z, R)\cap V(H)\not=\emptyset$ for all $H\in \mathcal{H}$: Since $|\mathscr{C}|\leq k-1$, $|Y|\leq 2(k-1)$. Let $X\coloneqq \pad(\bigcup\mathscr{C}, Y, 2\ceil{d/2}+1, \ceil{\ell/2})$. The following shows that $B_G(W, \lfloor d/2\rfloor)\cup B_{\bigcup\mathscr{C}\cup U}(X,R)$ meets every cycle in $G_r$ of length at least $\ell$, which implies that \ref{item:alternative:hitting-long-cycles-close-to-packing} holds. Consider any cycle $D$ in $G_r$ of length at least $\ell$. It may be assumed that $D$ has no vertex in $B_G(W, \lfloor d/2\rfloor)$.
Then $D\in \mathcal{D}$. Let $u\in B_U(Z, R)\cap V(H_D)$ as promised by the premise of the present case. Since $Z\subseteq V(\bigcup\mathscr{C})$, there exists $z\in Z$ such that $z=\proj_U(u)$. Since $u\in V(H_D)$, there exists $v\in V(D)$ and a path $q_0q_1\cdots q_s$ in $H_D$ such that $q_0=v$, $q_s=u$, and $s\leq \ceil{d/2}$. Then by \cref{clm:edges-of-H_D-have-big-span},
\begin{align*}
    \dist_{\bigcup\mathscr{C}}(\proj_U(v), z) &\leq \sum_{i=1}^s \dist_{\bigcup\mathscr{C}}(\proj_U(q_{i-1}), \proj_U(q_i))\\
    &= \sum_{i=1}^s \textstyle\spn(G_R, \bigcup\mathscr{C}, U, q_{i-1}q_i) \leq s\ell \leq \lceil d/2 \rceil \ell.
\end{align*}
Furthermore, since $z\in Z=B_{\bigcup\mathscr{C}}(Y, \floor{\ell/2})$,
\[\dist_{\bigcup\mathscr{C}}(\proj_U(v), Y)     \leq \dist_{\bigcup\mathscr{C}}(\proj_U(v), z) + \dist_{\bigcup\mathscr{C}}(z, Y) \leq \lceil d/2 \rceil \ell + \floor{\ell/2}\leq (2\ceil{d/2}+1)\ceil{\ell/2}.\]
Since $X=\pad(\bigcup\mathscr{C}, Y, 2\ceil{d/2}+1, \ceil{\ell/2})$, \cref{cor:def-of-pad} implies $\dist_{\bigcup\mathscr{C}}(\proj_U(v), X)\leq \ceil{d/2}$. Finally, since $v\in V(D) \subseteq V(G_r)$, $v\in B_U(\proj_U(v), r)$, thus $v\in B_{\bigcup\mathscr{C}\cup U}(X,r+\ceil{d/2})$. Consequently, $D$ has a vertex in $B_{\bigcup\mathscr{C}\cup U}(X,R)$ and \ref{item:alternative:hitting-long-cycles-close-to-packing} holds.

This concludes the proof of the lemma.
\end{proof}

We now prove a quantitative version of \cref{thm:main-in-intro}.

\begin{thm}\label{thm:main-technical-version}
Let $f:\N^2\to\N$ and $g:\N\to\N$ be the following functions:
\[f(k,\ell) = (51s(k)+104k-155)(2\ceil{\ell/2}+1)+k-1, \qquad g(d) = 21d.\]
For all positive integers $k,d,\ell$ with $\ell\geq 3$, for every graph $G$, either $G$ has a $d$-packing of $k$ cycles each of length at least $\ell$, or there exists $X\subseteq V(G)$ with $|X|\leq f(k, \ell)$ and $G-B_G(X, g(d))$ has no cycle of length at least $\ell$.
\end{thm}

\begin{proof}
If $G$ has a $d$-packing of $k$ cycles each of length at least $\ell$, then there is nothing to prove. Therefore we assume that $G$ does not have a $d$-packing of $k$ cycles each of length at least $\ell$.

Let \defin{$\mathcal{Z}$} be a maximal $d$-packing of cycles in $G$ such that $\ell \leq \len(D)\leq 6d+2$ for all $D\in \mathcal{Z}$. Let \defin{$\mathscr{C}$} be a maximal $2d$-packing of cycles in $G$ such that $\len(C)\geq \ell$ and the $C$-span of $G[B_G(V(C), d)]$ is at most $\ell$ for all $C\in \mathscr{C}$.

Note that $\mathcal{Z}$ and $\mathscr{C}$ have size at most $k-1$, and each may be the empty set.

\begin{clm}\label{clm:long-cycles-are-close-to-packing}
Every cycle in $G$ of length at least $\ell$ contains a vertex in $B_G(\bigcup \mathscr{C}, 4d)\cup B_G(\bigcup\mathcal{Z}, 4d)$.
\end{clm}

\begin{clmproof}
Assume for a contradiction that there exists a cycle $C$ in $G$ such that $C$ has length at least $\ell$ and contains no vertex in $B_G(\bigcup \mathscr{C}, 4d)\cup B_G(\bigcup\mathcal{Z}, 4d)$. Proceed by cases depending on the outcomes of \cref{lem:medium-cycle-or-small-span}.

\textbf{Case \ref{item:medium:medium-cycle-or-small-span}} there exists a cycle $C'$ in $G$ of length at least $\ell$ such that $V(C')\subseteq B_G(V(C), 3d)$ and the length of $C'$ is at most $6d+2$: Then $V(C') \cap B_G(\bigcup\mathcal{Z}, d)=\emptyset$, therefore $\mathcal{Z}\cup \set{C'}$ contradicts the maximality of $\mathcal{Z}$.

\textbf{Case \ref{item:small-span:medium-cycle-or-small-span}} there exists a cycle $C'$ in $G$ of length at least $\ell$ such that $V(C')\subseteq B_G(V(C), 2d)$ and the $C'$-span of $G[B_G(V(C'), d)]$ is at most $\ell$: Then $V(C')\cap B_G(\bigcup\mathscr{C}, 2d)=\emptyset$, therefore $\mathscr{C}\cup \set{C'}$ contradicts the maximality of $\mathscr{C}$.
\end{clmproof}

Let $\mathdefin{Z} \coloneqq  \textstyle B_G(\bigcup\mathcal{Z}, 4d)$ and let \mathdefin{$X_0$} be a set consisting of one vertex from each cycle in $\mathcal{Z}$. Then
\begin{equation}\label{eq:size-of-X_0}
    |X_0|\leq k-1.
\end{equation}
Furthermore, since each cycle in $\mathcal{Z}$ has length at most $6d+2$,
\begin{equation}\label{eq:Z-in-X_0-ball}
    Z \subseteq B_G(X_0, 7d+1).
\end{equation}

If $\mathscr{C}=\emptyset$, then \eqref{eq:size-of-X_0} and \eqref{eq:Z-in-X_0-ball} along with \cref{clm:long-cycles-are-close-to-packing} imply that the theorem holds. Hence from this point on we assume that $\mathscr{C}\not=\emptyset$ and thus, $\bigcup\mathscr{C}$-span is now well-defined for graphs that contain $\bigcup\mathscr{C}$ as a subgraph.

Let
\[\textstyle \mathdefin{G_d}\coloneqq G[B_G(V(\bigcup\mathscr{C}), d)].\]

\begin{clm}\label{clm:G_d-has-small-span}
The $\bigcup\mathscr{C}$-span of $G_d$ is at most $\ell$.
\end{clm}

\begin{clmproof}
Suppose $U$ is a $\bigcup\mathscr{C}$-supported BFS-spanning subgraph of $G_d$ and $e\in E(G_d)$. First consider the case that both endpoints of $e$ lie in $B_G(V(C), d)$ for some $C\in \mathscr{C}$. Let $G_C\coloneqq G[B_G(V(C), d)]$. Since $\mathscr{C}$ is a $2d$-packing in $G$, $S\coloneqq U[V(G_C)]$ is a $C$-supported BFS-spanning subgraph of $G_C$ such that $\spn(G_C, C, S, e) = \spn(G_d, \bigcup\mathscr{C}, U, e)$. Furthermore since $\spn(G_C, C, S, e)$ is finite and the $C$-span of $G_C$ is at most $\ell$ (by definition of $C\in \mathscr{C}$), $\spn(G_C, C, S, e) \leq \ell$. Thus $\spn(G_d, \bigcup\mathscr{C}, U, e) \leq \ell$. On the other hand, if $e$ is an edge between $B_G(V(C), d)$ and $B_G(V(D), d)$ for distinct $C,D\in\mathscr{C}$, then $\spn(G_d, \bigcup\mathscr{C}, U, e) = \infty$. It follows that the $(\bigcup \mathscr{C}, U)$-span of $G_d$ is at most $\ell$. Moreover, since $U$ was an arbitrary $\bigcup\mathscr{C}$-supported BFS-spanning subgraph of $G_d$, the $\bigcup\mathscr{C}$-span of $G_d$ is at most $\ell$, as desired.
\end{clmproof}

Let
\[\mathdefin{r}\coloneqq 6d,\]
\[\textstyle \mathdefin{G_r}\coloneqq G[B_G(V(\bigcup\mathscr{C}),r)],\]
\[\mathdefin{R}\coloneqq r+\ceil{d/2},\]
\[\textstyle \mathdefin{G_R}\coloneqq G[B_G(V(\bigcup\mathscr{C}), R)].\]
Let \defin{$U$} be a $\bigcup\mathscr{C}$-supported BFS-spanning subgraph of $G_R$.

Let \defin{$W$} be the set of all endpoints of edges $e\in E(G_R)$ such that $\spn(G_R, \bigcup\mathscr{C}, U, e)>\ell$. Since $\mathscr{C}$ is a $2d$-packing in $G_R$, every edge $e\in E(G[B_G(V(C), d)])$ has finite $(\bigcup\mathscr{C}, U)$-span for all $C\in \mathscr{C}$.  Thus \cref{clm:G_d-has-small-span} implies that every edge $e\in E(G[B_G(V(C), d)])$ satisfies $\spn(G_R, \bigcup\mathscr{C}, U, e) \leq \ell$.

Therefore for all $C\in \mathscr{C}$,
\begin{equation}\label{eq:d-1-balls-are-clean-of-W}
    \textstyle W\cap B_G(V(C), d-1) = \emptyset.
\end{equation}

Let
\[\mathdefin{I}\coloneqq \bigcup_{\set{C,D}\in \binom{\mathscr{C}}{2}}B_G(V(C), r)\cap B_G(V(D), r).\]

By \cref{clm:G_d-has-small-span} the $\bigcup\mathscr{C}$-span of $G_d$ is at most $\ell$. Since $R\geq r >d$, \cref{lem:big-span-and-intersecting-balls} implies that there exists $\mathdefin{X_1} \subseteq V(\bigcup\mathscr{C})$ such that
\begin{equation}\label{eq:size-of-X_1}
    |X_1|\leq (s(k)-1)(2\ceil{\ell/2}+1),
\end{equation}
and
\begin{equation}\label{eq:W-and-intersections-in-X_1-ball}
    W\cup I\subseteq B_G(W, r-d-1)
    \subseteq B_G(X_1, R+r+d).
\end{equation}

Since $\mathscr{C}$ is a $2d$-packing, it is a $d$-packing. Therefore since $r\geq \ceil{d/2}$, \cref{lem:hitting-long-cycles-close-to-packing} implies that there exists $\mathdefin{Y_2}\subseteq V(\bigcup\mathscr{C})$ with $|Y_2|\leq 2(k-1)$ such that, by letting
\[\textstyle\mathdefin{X_2}\coloneqq \pad(\bigcup\mathscr{C}, Y_2, 2\ceil{d/2}+1, \ceil{\ell/2}),\]
and
\[\mathdefin{L}\coloneqq B_G(W, \lfloor d/2 \rfloor) \cup B_{\bigcup\mathscr{C}\cup U}(X_2,R),\]
we have that
\begin{equation}\label{eq:L-hits-long-cycles-in-G_r}
    \textstyle \text{$G_r-L$ has no cycles of length at least $\ell$.}
\end{equation}
Then by \cref{cor:def-of-pad},
\begin{equation}\label{eq:size-of-X_2}
    |X_2|\leq |Y_2|(2\ceil{\ell/2}+1)\leq 2(k-1)(2\ceil{\ell/2}+1).
\end{equation}

Let
\[\mathdefin{F_{-1}} \coloneqq  \textstyle \bigcup\mathscr{C}-L,\]
\[\mathdefin{F_0} \coloneqq  \textstyle G-\big(Z\cup  B_G( V(\bigcup \mathscr{C}), r-2d)\big),\]
\[\mathdefin{F_0^-}\coloneqq \textstyle G-\big(Z\cup  B_G(V(\bigcup  \mathscr{C}), r-d)\big).\]
Then $F_0^-$ is an induced subgraph of $F_0$. Moreover by \cref{clm:long-cycles-are-close-to-packing},
\begin{equation}\label{eq:no-long-cycles-in-F_0}
    \text{$F_0$ and $F_0^-$ contain no cycle of length at least $\ell$.}
\end{equation}

Let $\mathdefin{p}\coloneqq |\mathscr{C}|$ and enumerate $\mathscr{C}$ as \defin{$C_1, C_2, \dots, C_p$}. For each $i\in [p]$ let
\[\mathdefin{F_i}\coloneqq G[B_G(V(C_i), r)]-L,\]
\[\mathdefin{F_i^-}\coloneqq G[B_G(V(C_i), r-d)]-L.\]
Then $F_i^-$ is an induced subgraph of $F_i$. Moreover, \eqref{eq:L-hits-long-cycles-in-G_r} guarantees that for all $i\in [p]$,

\begin{equation}\label{eq:no-long-cycles-in-F_i}
    \text{$F_i$ contains no cycle of length at least $\ell$.}
\end{equation}

Let
\begin{equation}\label{eq:def-of-M}
    \mathdefin{M} \coloneqq  B_G(X_0\cup X_1\cup X_2, R+r+d),
\end{equation}
then by \eqref{eq:Z-in-X_0-ball} and \eqref{eq:W-and-intersections-in-X_1-ball},
\begin{align}
    Z\cup L\cup I &\subseteq B_G(X_0, 7d+1)\cup B_G(W, \lfloor d/2 \rfloor) \cup B_{\bigcup\mathscr{C}\cup U}(X_2,R)\cup B_G(X_1, R+r+d)\notag\\
    &\subseteq B_G(X_0, 7d+1)\cup B_G(W, r-d-1) \cup B_G(X_2,R)\cup B_G(X_1, R+r+d)\notag\\
    &= B_G(X_0, 7d+1) \cup B_G(X_2,R)\cup B_G(X_1, R+r+d)\notag\\
    &\subseteq M.\label{eq:M-contains-Z-and-L-and-I}
\end{align}

For each $i\in [p]$, define $\mathdefin{U_i} \coloneqq  U[B_U(V(C_i), r-d)]$. Then $\bigcup_{i\in [p]}U_i$ is a forest and the sets $V(U_1), V(U_2), \dots, V(U_p)$ form a partition of $B_G(V(\bigcup\mathscr{C}), r-d)$. See \cref{fig:initial-summary} for a summary of the many key definitions given so far.

\begin{figure}[h]
\centering

\begin{tikzpicture}[smallMball/.style={circle, fill=\Mcol, fill opacity=0.7, inner sep=3pt}, Uline/.style={\Ucol, thick, draw, line cap=round},/tikz/xscale=1.8, /tikz/yscale=1.8]

% M

\node[smallMball] at (-0.75,0.3) {};

\node[smallMball, inner sep=5pt] at (-1.65,0.32) {};

\node[smallMball] at (-1.65,-0.7) {};

\node[smallMball] at (-2,1) {};

\node[smallMball, inner sep=6pt] at (-2.5,0.45) {};

\node[smallMball, inner sep=6pt] at (-0.5,-0.3) {};

\node[smallMball, inner sep=7pt] at (-2.4,-0.9) {};

\node[smallMball, inner sep=7pt] at (2.6,0.8) {};

\node[smallMball] at (1.9,0.25) {};

\node[smallMball] at (1.5,1.05) {};

\node[smallMball] at (0.7,-0.12) {};

\node[smallMball, inner sep=6pt] at (1.75,-0.25) {};

\node[smallMball, inner sep=6pt] at (0.7,0.85) {};

\node[smallMball, inner sep=9pt] at (-0.3,0.8) {};

\node[text=\Mcol, text opacity=0.7] at (2.95,0.98) {$M$};

% C_i
\def\ic{(-1.5,0)}

\path[draw, use Hobby shortcut, closed=true, thick, postaction={
    decorate,
    decoration={
      markings,
      mark=at position 0.2 with
      {
        \coordinate (x1);
      },
      mark=at position 0.6 with
      {
        \coordinate (x2);
      },
      mark=at position 0.05 with
      {
        \coordinate (x3);
      },
      mark=at position 0.4 with
      {
        \coordinate (x4);
      }
    }
  }]
(-1.9,0.1) .. (-1.75,0.2) .. (-1.5,0.4) .. (-1.5,-0.4);

\draw[opacity=0, postaction={
    decorate,
    decoration={
      markings,
      mark=at position 0.6 with
      {
        \coordinate (x5);
      }
    }
  }] \ic circle (0.48);

\draw[dashed, opacity=0.5] \ic circle (0.8);

\draw[dashed, postaction={
    decorate,
    decoration={
      markings,
      mark=at position 0.18 with
      {
        \coordinate (x6);
      },
      mark=at position 0.27 with
      {
        \coordinate (x7);
      },
      mark=at position 0.35 with
      {
        \coordinate (x8);
      },
      mark=at position 0.5 with
      {
        \coordinate (x9);
      },
      mark=at position 0.6 with
      {
        \coordinate (x10);
      },
      mark=at position 0.75 with
      {
        \coordinate (x11);
      },
      mark=at position 0.85 with
      {
        \coordinate (x12);
      },
    }
  }, opacity=0.5] \ic circle (1);

\draw[dashed, opacity=0.5] \ic circle (1.2);

\node[text opacity=0.7, circle, fill=white, fill opacity=1, inner sep=0pt] at (-2.6,-0.42) {$F_i$};

\node[text opacity=0.7, circle, fill=white, fill opacity=1, inner sep=0pt] at (-2,-0.6) {$F_i^-$};

\draw[->, thick, opacity=0.5] (x5) to (x10);

\coordinate (z1) at ($(x5)!0.1cm!-90:(x10)$);

\coordinate (z2) at ($($(x10)-(x5)+(z1)$)+0.35*($($(x10)-(x5)+(z1)$)-(z1)$)$);

\draw[->, thick, opacity=0.5] (z1) to (z2);

\node at (-1.5,0) {$C_i$};

% U_i

\path[Uline, postaction={
    decorate,
    decoration={
      markings,
      mark=at position 0.15 with
      {
        \coordinate (u1);
      },
      mark=at position 0.5 with
      {
        \coordinate (u2);
      }
    }
  }] (x2) to [out=-70,in=0] (x11);

\path[Uline] (u1) to [out=-90,in=90] (-1.7, -0.9);

\path[Uline] (u2) to [out=-90,in=90] (x12);

\path[Uline, postaction={
decorate,
decoration={
  markings,
  mark=at position 0.4 with
  {
    \coordinate (u3);
  },
  mark=at position 0.7 with
  {
    \coordinate (u4);
  }
}
}] (x3) to [out=120,in=100] (-2.3,-0.2);

\path[Uline] (u4) to [out=-120,in=90] (x9);

\path[Uline, postaction={
decorate,
decoration={
  markings,
  mark=at position 0.5 with
  {
    \coordinate (u5);
  }
}
}] (u3) to [out=-150,in=230] (x8);

\path[Uline] (u5) to [out=0,in=-90] (-1.8,0.7);

\path[Uline, postaction={
decorate,
decoration={
  markings,
  mark=at position 0.3 with
  {
    \coordinate (u6);
  }
}
}] (x4) to [out=0,in=45] (-0.8,-0.4);

\path[Uline] (u6) to [out=20,in=-180] (-0.6,0.4);

\path[Uline] (u6) to [out=0,in=-180] (-0.6,0);

\path[Uline, postaction={
decorate,
decoration={
  markings,
  mark=at position 0.3 with
  {
    \coordinate (u7);
  }
}
}] (x1) to [out=0,in=-180] (x6);

\path[Uline] (u7) to [out=80,in=-135] (x7);

\node[\Ucol] at (-1.05, 0.45) {$U_i$};

% C_j
\def\jc{(1.5,0.5)}

\path[draw, use Hobby shortcut, closed=true, thick, postaction={
    decorate,
    decoration={
      markings,
      mark=at position 0 with
      {
        \coordinate (y1);
      },
      mark=at position 0.2 with
      {
        \coordinate (y2);
      },
      mark=at position 0.4 with
      {
        \coordinate (y3);
      },
      mark=at position 0.6 with
      {
        \coordinate (y4);
      },
      mark=at position 0.8 with
      {
        \coordinate (y5);
      },
    }
  }]
(1.28,0.85) .. (1.78,0.95) .. (1.68,0.55) .. (1.88,0.25) .. (1.18,0.55);

\node at \jc {$C_j$};

\draw[dashed, opacity=0.5] \jc circle (0.8);

\draw[dashed, postaction={
    decorate,
    decoration={
      markings,
      mark=at position 0 with
      {
        \coordinate (v1);
      },
      mark=at position 0.1 with
      {
        \coordinate (v2);
      },
      mark=at position 0.15 with
      {
        \coordinate (v3);
      },
      mark=at position 0.3 with
      {
        \coordinate (v4);
      },
      mark=at position 0.4 with
      {
        \coordinate (v5);
      },
      mark=at position 0.45 with
      {
        \coordinate (v6);
      },
      mark=at position 0.6 with
      {
        \coordinate (v7);
      },
      mark=at position 0.65 with
      {
        \coordinate (v8);
      },
      mark=at position 0.7 with
      {
        \coordinate (v9);
      },
      mark=at position 0.85 with
      {
        \coordinate (v10);
      },
    }
  }, opacity=0.5] \jc circle (1);

\draw[dashed, opacity=0.5] \jc circle (1.2);

% U_j

\path[Uline, postaction={
decorate,
decoration={
  markings,
  mark=at position 0.15 with
  {
    \coordinate (m3);
  },
  mark=at position 0.4 with
  {
    \coordinate (m1);
  },
  mark=at position 0.7 with
  {
    \coordinate (m2);
  }
}
}] (y1) to [out=135,in=45] (v6);

\path[Uline] (m3) to [out=180,in=0] (0.8,0.5);

\path[Uline] (m1) to [out=135,in=-45] (v4);

\path[Uline] (m2) to [out=90,in=0] (v5);

\path[Uline] (y5) to [out=135,in=45] (v7);

\path[Uline, postaction={
decorate,
decoration={
  markings,
  mark=at position 0.4 with
  {
    \coordinate (m4);
  }
}
}] (y5) to [out=-90,in=45] (v9);

\path[Uline] (m4) to [out=180,in=90] (v8);

\path[Uline, postaction={
decorate,
decoration={
  markings,
  mark=at position 0.3 with
  {
    \coordinate (m5);
  }
}
}] (y4) to [out=-90,in=90] (v10);

\path[Uline] (m5) to [out=-45,in=0] (1.5,-0.4);

\path[Uline, postaction={
decorate,
decoration={
  markings,
  mark=at position 0.2 with
  {
    \coordinate (m6);
  },
  mark=at position 0.7 with
  {
    \coordinate (m7);
  }
}
}] (y2) to [out=90,in=-90] (v3);

\path[Uline] (m6) to [out=45,in=-70] (1.5,1.4);

\path[Uline] (m7) to [out=45,in=-135] (v2);

\path[Uline, postaction={
decorate,
decoration={
  markings,
  mark=at position 0.5 with
  {
    \coordinate (m8);
  }
}
}] (y3) to [out=60,in=-90] (2.1,0.9);

\path[Uline, postaction={
decorate,
decoration={
  markings,
  mark=at position 0.6 with
  {
    \coordinate (m9);
  }
}
}] (m8) to [out=45,in=-180] (2.3,0.2);

\path[Uline] (m9) to [out=45,in=-135] (v1);

\node[text=\Ucol, circle, fill=white, fill opacity=1, inner sep=0pt] at (2.3,-0.07) {$U_j$};

% F_0

\node[text opacity=0.7, circle, fill=white, fill opacity=1, inner sep=0pt] at (-0.36,0.33) {$F_0$};

\node[text opacity=0.7] at (-0.07,-0.07) {$F_0^-$};

\draw[<->, thick, opacity=0.5] ($(-1.5,0)!0.8cm!(1.5,0.5)$) to ($(1.5,0.5)!0.8cm!(-1.5,0)$);

\draw[<->, thick, opacity=0.5] ($($(-1.5,0)!1cm!(1.5,0.5)$)!0.15cm!90:(-1.5,0)$) to ($($(1.5,0.5)!1cm!(-1.5,0)$)!0.15cm!-90:(1.5,0.5)$);

\end{tikzpicture}

\caption{Simplified drawing showing the interplay between many of the key definitions so far.}
\label{fig:initial-summary}
\end{figure}
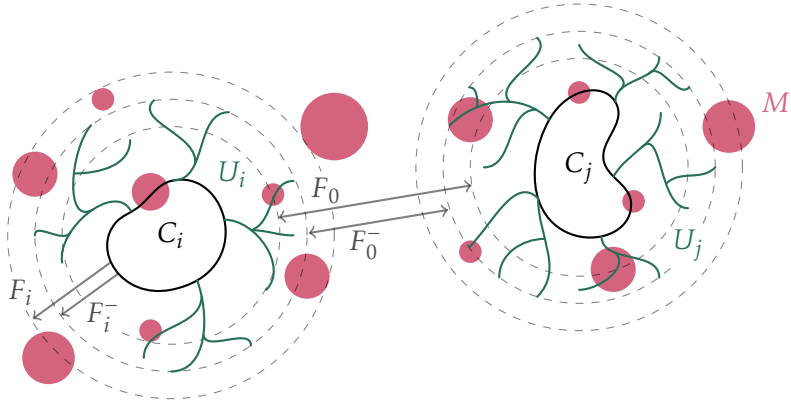

For each pair $i,j\in [p]$, a path $P\coloneqq v_0 v_1\cdots v_s$ in $G$ with $s\geqslant 2$ is an \defin{$(i,j)$-ear} if
\begin{itemize}
    \item $v_0\in V(U_i)$ and $v_s\in V(U_j)$;
    \item $v_1\cdots v_{s-1}$ is a path in $F_0^-$.
\end{itemize}
An \defin{ear} is an $(i,j)$-ear for some $i,j\in [p]$.
The \defin{first leg} of $P$ is the path $P_1\coloneqq U\angle{v_0}$, and the \defin{second leg} of $P$ is the path $P_2\coloneqq U\angle{v_s}$. Therefore $P_1$ is the unique path in $U_i$ between $v_0$ and $\proj_U(v_0)$, and $P_2$ is the unique path in $U_j$ between $v_s$ and $\proj_U(v_s)$. Note that $V(P_1)\cap V(P_2)$ may be non-empty for $(i,j)$-ears where $i=j$. Let
\[\mathdefin{\proj_U(P)}\coloneqq \set{\proj_U(v_0), \proj_U(v_s)}.\]
Then $\proj_U(P)$ is a non-empty subset of $V(\bigcup\mathscr{C})$. \cref{fig:ears} shows typical examples of ears and their projections. Define the walk 
\[\mathdefin{W_P} \coloneqq  P_1\cup P \cup P_2.\]
Note that the first and second legs of $P$ have exactly $r-d$ edges, whereas $P$ may be arbitrarily long. $P$ is said to be \defin{admissible} if $V(W_P)\cap B_G(M, d) = \emptyset$.

\begin{figure}[h]
\centering

\begin{tikzpicture}[Uline/.style={\Ucol, thick, draw, line cap=round}, Eline/.style={\Ecol, thick, draw, line cap=round}, proj/.style={circle, fill=\projcol, inner sep=1.2pt},/tikz/xscale=1.8, /tikz/yscale=1.8]

% C_i
\def\ic{(-1.5,0)}

\path[draw, use Hobby shortcut, closed=true, thick, postaction={
    decorate,
    decoration={
      markings,
      mark=at position 0.2 with
      {
        \coordinate (x1);
      },
      mark=at position 0.6 with
      {
        \coordinate (x2);
      },
      mark=at position 0.05 with
      {
        \coordinate (x3);
      },
      mark=at position 0.4 with
      {
        \coordinate (x4);
      },
      mark=at position 0.85 with
      {
        \coordinate (z);
      }
    }
  }]
(-1.9,0.1) .. (-1.75,0.2) .. (-1.5,0.4) .. (-1.5,-0.4);

\draw[opacity=0, postaction={
    decorate,
    decoration={
      markings,
      mark=at position 0.6 with
      {
        \coordinate (x5);
      }
    }
  }] \ic circle (0.48);

\draw[dashed, postaction={
    decorate,
    decoration={
      markings,
      mark=at position 0.18 with
      {
        \coordinate (x6);
      },
      mark=at position 0.27 with
      {
        \coordinate (x7);
      },
      mark=at position 0.35 with
      {
        \coordinate (x8);
      },
      mark=at position 0.5 with
      {
        \coordinate (x9);
      },
      mark=at position 0.6 with
      {
        \coordinate (x10);
      },
      mark=at position 0.75 with
      {
        \coordinate (x11);
      },
      mark=at position 0.85 with
      {
        \coordinate (x12);
      },
    }
  }, opacity=0.5] \ic circle (1);

\node at (-1.5,0) {$C_i$};

% U_i

\path[Uline, postaction={
    decorate,
    decoration={
      markings,
      mark=at position 0.15 with
      {
        \coordinate (u1);
      },
      mark=at position 0.5 with
      {
        \coordinate (u2);
      }
    }
  }] (x2) to [out=-70,in=0] (x11);

\path[Uline] (u1) to [out=-90,in=90] (-1.7, -0.9);

\path[Uline] (u2) to [out=-90,in=90] (x12);

\path[Uline, postaction={
decorate,
decoration={
  markings,
  mark=at position 0.4 with
  {
    \coordinate (u3);
  },
  mark=at position 0.7 with
  {
    \coordinate (u4);
  }
}
}] (x3) to [out=120,in=100] (-2.3,-0.2);

\path[Uline] (u4) to [out=-120,in=90] (x9);

\path[Uline, postaction={
decorate,
decoration={
  markings,
  mark=at position 0.5 with
  {
    \coordinate (u5);
  }
}
}] (u3) to [out=-150,in=230] (x8);

\path[Uline] (u5) to [out=0,in=-90] (-1.8,0.7);

\path[Uline, postaction={
decorate,
decoration={
  markings,
  mark=at position 0.3 with
  {
    \coordinate (u6);
  }
}
}] (x4) to [out=0,in=45] (-0.8,-0.4);

\path[Uline] (u6) to [out=20,in=-180] (-0.6,0.4);

\path[Uline] (u6) to [out=0,in=-180] (-0.6,0);

\path[Uline, postaction={
decorate,
decoration={
  markings,
  mark=at position 0.3 with
  {
    \coordinate (u7);
  }
}
}] (x1) to [out=0,in=-180] (x6);

\path[Uline] (u7) to [out=80,in=-135] (x7);

\path[Uline, postaction={
decorate,
decoration={
  markings,
  mark=at position 0.7 with
  {
    \coordinate (u8);
  }
}
}] (z) to [out=-135,in=135] (-2,-0.85);

\path[Uline] (u8) to [out=90,in=0] (-2.35,-0.5);

\node[\Ucol] at (-1.05, 0.45) {$U_i$};

% C_j
\def\jc{(1.5,0.5)}

\path[draw, use Hobby shortcut, closed=true, thick, postaction={
    decorate,
    decoration={
      markings,
      mark=at position 0 with
      {
        \coordinate (y1);
      },
      mark=at position 0.2 with
      {
        \coordinate (y2);
      },
      mark=at position 0.4 with
      {
        \coordinate (y3);
      },
      mark=at position 0.6 with
      {
        \coordinate (y4);
      },
      mark=at position 0.8 with
      {
        \coordinate (y5);
      },
    }
  }]
(1.28,0.85) .. (1.78,0.95) .. (1.68,0.55) .. (1.88,0.25) .. (1.18,0.55);

\node at \jc {$C_j$};

\draw[dashed, postaction={
    decorate,
    decoration={
      markings,
      mark=at position 0 with
      {
        \coordinate (v1);
      },
      mark=at position 0.1 with
      {
        \coordinate (v2);
      },
      mark=at position 0.15 with
      {
        \coordinate (v3);
      },
      mark=at position 0.3 with
      {
        \coordinate (v4);
      },
      mark=at position 0.4 with
      {
        \coordinate (v5);
      },
      mark=at position 0.45 with
      {
        \coordinate (v6);
      },
      mark=at position 0.6 with
      {
        \coordinate (v7);
      },
      mark=at position 0.65 with
      {
        \coordinate (v8);
      },
      mark=at position 0.7 with
      {
        \coordinate (v9);
      },
      mark=at position 0.85 with
      {
        \coordinate (v10);
      },
    }
  }, opacity=0.5] \jc circle (1);

% \draw[dashed, opacity=0.5] \jc circle (1.2);

% U_j

\path[Uline, postaction={
decorate,
decoration={
  markings,
  mark=at position 0.15 with
  {
    \coordinate (m3);
  },
  mark=at position 0.4 with
  {
    \coordinate (m1);
  },
  mark=at position 0.7 with
  {
    \coordinate (m2);
  }
}
}] (y1) to [out=135,in=45] (v6);

\path[Uline] (m3) to [out=180,in=0] (0.8,0.5);

\path[Uline] (m1) to [out=135,in=-45] (v4);

\path[Uline] (m2) to [out=90,in=0] (v5);

\path[Uline] (y5) to [out=135,in=45] (v7);

\path[Uline, postaction={
decorate,
decoration={
  markings,
  mark=at position 0.4 with
  {
    \coordinate (m4);
  }
}
}] (y5) to [out=-90,in=45] (v9);

\path[Uline] (m4) to [out=180,in=90] (v8);

\path[Uline, postaction={
decorate,
decoration={
  markings,
  mark=at position 0.3 with
  {
    \coordinate (m5);
  }
}
}] (y4) to [out=-90,in=90] (v10);

\path[Uline] (m5) to [out=-45,in=0] (1.5,-0.4);

\path[Uline, postaction={
decorate,
decoration={
  markings,
  mark=at position 0.2 with
  {
    \coordinate (m6);
  },
  mark=at position 0.7 with
  {
    \coordinate (m7);
  }
}
}] (y2) to [out=90,in=-90] (v3);

\path[Uline] (m6) to [out=45,in=-70] (1.5,1.4);

\path[Uline] (m7) to [out=45,in=-135] (v2);

\path[Uline, postaction={
decorate,
decoration={
  markings,
  mark=at position 0.5 with
  {
    \coordinate (m8);
  }
}
}] (y3) to [out=60,in=-90] (2.1,0.9);

\path[Uline, postaction={
decorate,
decoration={
  markings,
  mark=at position 0.6 with
  {
    \coordinate (m9);
  }
}
}] (m8) to [out=45,in=-180] (2.3,0.2);

\path[Uline] (m9) to [out=45,in=-135] (v1);

\node[text=\Ucol, circle, fill=white, fill opacity=1, inner sep=0pt] at (2.1,0) {$U_j$};

% F_0^-

\node[text opacity=0.7] at (0,0.1) {$F_0^-$};

\draw[<->, thick, opacity=0.5] ($(-1.5,0)!1cm!(1.5,0.5)$) to ($(1.5,0.5)!1cm!(-1.5,0)$);

% Ears

\path[Eline] (x6) to [out=0,in=0] (-1.2,1.4) to [out=180,in=45] (x7);

\node[\Ecol] at (-1.61,1.3) {$E$};

\path[Eline] (v5) to [out=180,in=-135] (0.8,1.4) to [out=45,in=90] (v3);

\node[\Ecol] at (0.55,1.4) {$Q$};

\path[Eline] (x12) to [out=-90,in=-135] (-0.2,-0.5) to [out=45,in=-135] (v9);

\node[\Ecol] at (0,-0.57) {$P$};

% Projections
\node[proj] at (x1) {};

\node[\projcol] at (-1.63,0.5) {$u$};

\node[proj] at (y1) {};

\node[\projcol] at (1.25,1.02) {$v$};

\node[proj] at (y2) {};

\node[\projcol] at (1.93,0.94) {$w$};

\node[proj] at (x2) {};

\node[\projcol] at (-1.15,-0.4) {$x$};

\node[proj] at (y5) {};

\node[\projcol] at (1.07,0.14) {$y$};

\end{tikzpicture}

\caption{$E$ is an $(i,i)$-ear with $\proj_U(E) = \set{x}$, $Q$ is a $(j,j)$-ear with $\proj_U(Q) = \set{v,w}$, and $P$ is an $(i,j)$-ear with $\proj_U(P) = \set{x,y}$.}
\label{fig:ears}
\end{figure}
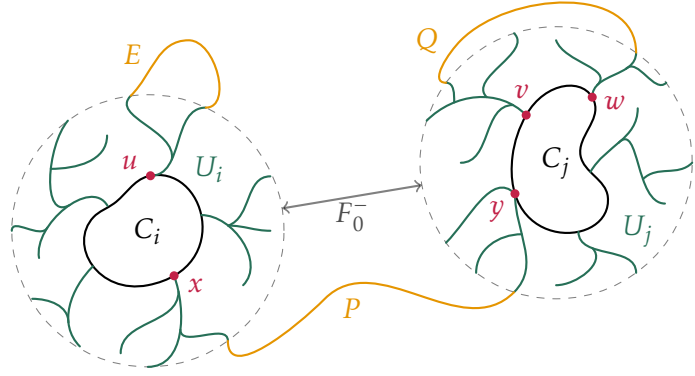

Suppose $P\coloneqq v_0 v_1\cdots v_s$ is an admissible $(i,j)$-ear and let $P_0 \coloneqq  v_1 \cdots v_{s-1}$. For each $m\in \set{-1,0,1,\dots, p}$, let
\[
\mathdefin{\Psi_m(P)} \coloneqq  \begin{cases}
    \emptyset & \text{if $m\not\in \set{-1,0,i,j}$,}\\
    B_{\bigcup\mathscr{C}-L}(\proj_U(P), d\ell) & \text{if $m=-1$,}\\
    B_G(V(P_0), d) & \text{if $m=0$,}\\
    B_G(V(W_P)\cap V(U_m), d) & \text{if $m\in \set{i,j}$.}
\end{cases}
\]

\begin{comment}
%%%%%%%%%%
% New definition of $\Psi_m$ is compact and equivalent to the old one.
%%%%%%%%%%

Suppose $P\coloneqq v_0 v_1\cdots v_s$ is an admissible $(i,j)$-ear with first leg $P_1$ and second leg $P_2$, and let $P_0 \coloneqq  v_1 \cdots v_{s-1}$. For each $m\in \set{-1,0,1,\dots, p}$, let
\[
\mathdefin{\Psi_m(P)} \coloneqq  \begin{cases}
    \emptyset & \text{if $m\not\in \set{-1,0,i,j}$,}\\
    B_{\bigcup\mathscr{C}-L}(\proj_U(P), d\ell) & \text{if $m=-1$,}\\
    B_G(V(P_0), d) & \text{if $m=0$,}\\
    B_G(V(P_1), d) & \text{if $m=i$ and $m\not=j$,}\\
    B_G(V(P_2), d) & \text{if $m\not=i$ and $m=j$,}\\
    B_G(V(P_1\cup P_2), d) & \text{if $m=i$ and $m=j$.}
\end{cases}
\]
\end{comment}

\begin{clm}\label{clm:components-of-Psi}
If $P$ is an admissible ear, then
\begin{enumerate}
    \item $G[\Psi_{-1}(P)]$ is a subgraph of $F_{-1}$ and has at most two components;\label{item:-1:components-of-Psi}
    \item $G[\Psi_0(P)]$ is a connected subgraph of $F_0$;\label{item:0:components-of-Psi}
    \item for all $m\in [p]$, $G[\Psi_m(P)]$ is a subgraph of $F_m$ and has $c_m$ components where $c_m\in \set{0,1,2}$. Moreover, $\sum_{m\in [p]}c_m \leq 2$.\label{item:positive:components-of-Psi}
\end{enumerate}
\end{clm}

\begin{clmproof}
Write $P=v_0v_1\cdots v_s$ and let $i,j\in [p]$ such that $P$ is an $(i,j)$-ear. Let $P_1$ and $P_2$ be the first and second legs of $P$ respectively, and let $P_0 \coloneqq  v_1\cdots v_{s-1}$. Let $u_0\coloneqq \proj_U(v_0)$ and $u_s\coloneqq \proj_U(v_s)$. Recall that $Z\cup L\subseteq M$ by \eqref{eq:M-contains-Z-and-L-and-I}.

To prove \ref{item:-1:components-of-Psi}, it suffices to show that $B_{C_i-L}(u_0, d\ell)$ is a connected set of vertices in $C_i-L$, and $B_{C_j-L}(u_s, d\ell)$ is a connected set of vertices in $C_j-L$. However this is trivial since each is a ball centred at some vertex.

Next, we prove \ref{item:0:components-of-Psi}. Since $P$ is admissible, $B_G(V(W_P), d) \cap M=\emptyset$. Thus $P_0\subseteq W_P$ and $Z\subseteq M$ together imply $B_G(V(P_0), d) \cap Z=\emptyset$. Furthermore since $P_0\subseteq F_0^-$, $V(P_0) \cap B_G(V(\bigcup\mathscr{C}), r-d)=\emptyset$. Hence $B_G(V(P_0), d)$ and $Z\cup B_G(V(\bigcup\mathscr{C}), r-2d)$ are disjoint. Therefore, since $P_0$ is connected, $G[B_G(V(P_0), d)]$ is a connected subgraph of $G-(Z\cup B_G(V(\bigcup\mathscr{C}), r-2d))$. In other words, $G[\Psi_0(P)]$ is a connected subgraph of $F_0$.

It remains to prove \ref{item:positive:components-of-Psi}. Let $m\in [p]$. Note that $m\not=0$. If $m\not\in \set{i,j}$, then $\Psi_m(P) = \emptyset$, implying that $G[\Psi_m(P)]\subseteq F_m$ and $c_m=0$. Therefore it suffices to show that $B_G(V(P_1), d)$ is a connected set of vertices in $F_i$, and $B_G(V(P_2), d)$ is a connected set of vertices in $F_j$.
We only prove the first statement as the second is symmetric. Since $P$ is admissible and $L\subseteq M$, $B_G(V(W_P), d)\cap L=\emptyset$, thus $B_G(V(P_1), d)\cap L=\emptyset$. Furthermore, since $V(P_1)\subseteq V(U_i) \subseteq B_G(V(C_i), r-d)$, $B_G(V(P_1), d)\subseteq B_G(V(C_i), r)\setminus L$. Therefore, since $P_1$ is connected, $B_G(V(P_1), d)$ is a connected set of vertices in $G[B_G(V(C_i), r)]-L$. In other words, $B_G(V(P_1), d)$ is a connected set of vertices in $F_i$ as desired.
\end{clmproof}

\begin{clm}\label{clm:dist-and-proj}
If $P$ and $Q$ are admissible ears such that $\Psi_{-1}(P) \cap \Psi_{-1}(Q)=\emptyset$, then
\[\dist_{\bigcup\mathscr{C}}(\proj_U(P), \proj_U(Q)) > d\ell.\]
\end{clm}

\begin{clmproof} Let $P=v_0v_1\cdots v_s$ and let $i,j\in [p]$ such that $P$ is an $(i,j)$-ear. Let $P_1$ and $P_2$ be the first and second legs of $P$ respectively, and let $P_0 \coloneqq  v_1\cdots v_{s-1}$. Similarly, write $Q=w_0w_1\cdots w_h$ and let $i',j'\in [p]$ such that $Q$ is an $(i',j')$-ear. Let $Q_1$ and $Q_2$ be the first and second legs of $Q$ respectively, and let $Q_0 \coloneqq  w_1\cdots w_{h-1}$.

Proceed by contraposition and suppose that $\dist_{\bigcup\mathscr{C}}(\proj_U(P), \proj_U(Q)) \leq d\ell$. Let $S$ be a shortest path in $\bigcup\mathscr{C}$ between $\proj_U(P)$ and $\proj_U(Q)$, and let $C\in \mathscr{C}$ such that $S\subseteq C$. The following shows that $S\subseteq C-L$, which implies that $\dist_{\bigcup\mathscr{C}-L}(\proj_U(P), \proj_U(Q))\leq d\ell$, hence $\Psi_{-1}(P)\cap \Psi_{-1}(Q)\not=\emptyset$ as required. As a first step, we show that the components of $C[V(C)\cap L]$ each have at least $\gamma\coloneqq \min\set{\len(C), d\ell+2}$ vertices. 
Recall that $X_2=\pad(\bigcup\mathscr{C}, Y_2, \alpha, \beta)$ (where $\alpha=2\ceil{d/2}+1$ and $\beta=\ceil{\ell/2}$). Since $X_2\subseteq V(\bigcup\mathscr{C})$ and the components of $\bigcup\mathscr{C}\cup U$ are the graphs $C_1\cup U[B_U(V(C_1), R)],\dots, C_p\cup U[B_U(V(C_p), R)]$, for all $t \geq 0$ we have
\begin{equation}\label{eq:C-intersect-L}
    V(C)\cap B_{\bigcup\mathscr{C}\cup U}(X_2, t) = V(C)\cap B_{\bigcup\mathscr{C}}(X_2, t) = B_C(V(C)\cap X_2, t).
\end{equation}
By \cref{cor:def-of-pad} and $2(\alpha\beta+\lfloor\alpha/2\rfloor)+1\geq (d+1)(\ell+1) > d\ell+2$, the components of $C[V(C)\cap B_{\bigcup\mathscr{C}}(X_2, \floor{\alpha/2})]$ each have at least $\min\set{\len(C),d\ell+2} = \gamma$ vertices. Then \eqref{eq:C-intersect-L} implies the components of $C[B_C(V(C)\cap X_2, \floor{\alpha/2})]$ each have at least $\gamma$ vertices. Since $R\geq\floor{\alpha/2}$, the components of $C[B_C(V(C)\cap X_2, R)]$ each have at least $\gamma$ vertices, hence the components of $C[V(C)\cap B_{\bigcup\mathscr{C}\cup U}(X_2, R)]$ each have at least $\gamma$ vertices by \eqref{eq:C-intersect-L}. Now recall that $L=B_G(W, \floor{d/2})\cup B_{\bigcup\mathscr{C}\cup U}(X_2, R)$. Since $\floor{d/2}\leq d-1$, \eqref{eq:d-1-balls-are-clean-of-W} implies $B_G(W, \floor{d/2})\cap V(C)=\emptyset$, thus $V(C)\cap L = V(C)\cap B_{\bigcup\mathscr{C}\cup U}(X_2, R)$. It follows that the components of $C[V(C)\cap L]$ each have at least $\gamma$ vertices, as required. Next, since $P$ and $Q$ are admissible and $L\subseteq M$ by \eqref{eq:M-contains-Z-and-L-and-I}, the set of endpoints of $S$ is disjoint from $L$. Assume for a contradiction that $V(S)\cap L\not=\emptyset$. Then there exists a component of $C[V(C)\cap L]$ that is a subgraph of $S$, implying $\len(S)=|V(S)|-1 \geq \min\set{\len(C)-1, d\ell+1}$. However, by definition of $S$, $\len(S) \leq \min\set{\len(C)/2, d\ell} < \min\set{\len(C)-1,d\ell+1}$, a contradiction. Hence $V(S)\cap L=\emptyset$, implying $S\subseteq C-L$ as desired.
\end{clmproof}

\begin{clm}\label{clm:dist-and-W}
If $P$ and $Q$ are admissible ears such that $\Psi_m(P) \cap \Psi_m(Q)=\emptyset$ for all $m\in \set{0,1, \dots, p}$, then
\[\dist_G(V(W_P), V(W_Q)) > d.\]
\end{clm}

\begin{clmproof}
Write $P=v_0v_1\cdots v_s$ and let $i,j\in [p]$ such that $P$ is an $(i,j)$-ear. Let $P_1$ and $P_2$ be the first and second legs of $P$ respectively, and let $P_0 \coloneqq  v_1\cdots v_{s-1}$. Similarly, write $Q=w_0w_1\cdots w_h$ and let $i',j'\in [p]$ such that $Q$ is an $(i',j')$-ear. Let $Q_1$ and $Q_2$ be the first and second legs of $Q$ respectively, and let $Q_0 \coloneqq  w_1\cdots w_{h-1}$.

Proceed by contraposition and suppose that $\dist_G(V(W_P), V(W_Q)) \leq d$. Let $v\in V(W_P)$ and $w\in V(W_Q)$ such that $\dist_G(v,w) \leq d$. Then either \textbf{(1)} $v\in V(P_0)$ or $w\in V(Q_0)$, or \textbf{(2)} $v\not\in V(P_0)$ and $w\not\in V(Q_0)$.

\textbf{Case (1)} $v\in V(P_0)$ or $w\in V(Q_0)$: Without loss of generality assume the latter. Then it is immediate that $v\in B_G(w, d) \subseteq \Psi_0(Q)$, so it suffices to show that $v\in \Psi_0(P)$. If $v\in V(P_0)$, then $v\in \Psi_0(P)$. Otherwise $v$ lies in the first or second leg of $P$. Without loss of generality $v\in V(P_1)$. Let $S\coloneqq v P_1 v_0$. Recall that $w\in V(Q_0) \subseteq V(F_0^-)$. Also, recall that $P_1$ is a shortest path in $G$ between $V(\bigcup\mathscr{C})$ and $v_0$, and $\len(P_1)=r-d$, thus $B_G(v,\len(S)) \subseteq B_G(V(\bigcup\mathscr{C}), r-d)$. Hence $B_G(v, \len(S))$ is disjoint from $V(F_0^-)$. Therefore since $w\in B_G(v, d)\cap V(F_0^-)$, $\len(S)\leq d-1$. Then $S\cup v_0v_1$ is a path between $v$ and $V(P_0)$ of length at most $d$, which implies that $v\in B_G(V(P_0), d) = \Psi_0(P)$, as desired.

\textbf{Case (2)} $v\not\in V(P_0)$ and $w\not\in V(Q_0)$: Without loss of generality assume that $v\in V(P_1)$ and $w\in V(Q_1)$. It is immediate that $w\in B_G(v, d) \subseteq \Psi_i(P)$, so it suffices to show that $w\in \Psi_i(Q)$. By \cref{clm:components-of-Psi} $\Psi_i(P) \subseteq V(F_i)$, thus $w\in B_G(V(C_i), r)$. This along with $w\in V(Q_1) \subseteq V(U_{i'}) \subseteq B_G(V(C_{i'}), r)$ shows that $w\in B_G(V(C_i), r)\cap B_G(V(C_{i'}), r)$. Now since $Q$ is admissible and $w\in V(W_Q)$, $w\not\in M$. In particular $w\not\in I$ by \eqref{eq:M-contains-Z-and-L-and-I}. Consequently, $i=i'$. Hence $w \in B_G(V(Q_1), d) \subseteq \Psi_{i'}(Q) = \Psi_{i}(Q)$, as desired.
\end{clmproof}

\begin{clm}\label{clm:hitting-Phi-and-Psi}
If $P$ is an admissible ear and there exists sets $S\subseteq V(\bigcup\mathscr{C})$ and $S'\subseteq V(G)$ such that $S\cap \Psi_{-1}(P)\not=\emptyset$ or $S'\cap \bigcup_{m=0}^p\Psi_m(P)\not=\emptyset$, then $B_G(\pad(\bigcup\mathscr{C}, S, 2d, \ceil{\ell/2}) \cup S',r) \cap V(P)\not=\emptyset$.
\end{clm}

\begin{clmproof}
First suppose that $S\cap \Psi_{-1}(P)\not=\emptyset$. Then $S\cap B_{\bigcup\mathscr{C}}(\proj_U(P), d\ell)\not=\emptyset$. In other words, there exists $u\in \proj_U(P)$ such that $\dist_{\bigcup\mathscr{C}}(u, S) \leq d\ell \leq 2d \ceil{\ell/2}$. Then \cref{cor:def-of-pad} implies $\dist_{\bigcup\mathscr{C}}(u, \pad(\bigcup\mathscr{C}, S, 2d, \ceil{\ell/2}))\leq d$. Moreover, since $B_G(u, r-d)\cap V(P)\not=\emptyset$, $B_G(\pad(\bigcup\mathscr{C}, S, 2d, \ceil{\ell/2}), r) \cap V(P)\not=\emptyset$ as required.

Next, suppose that $S'\cap \bigcup_{m=0}^p\Psi_m(P)\not=\emptyset$. Since $\bigcup_{m=0}^p\Psi_m(P)\subseteq B_G(V(W_P), d)$, $S\cap B_G(V(W_P), d) \not=\emptyset$. Then since $V(W_P)\subseteq B_G(V(P), r-d)$, $S'\cap B_G(V(P), r)\not=\emptyset$, hence $B_G(S', r)\cap V(P)\not=\emptyset$ as required.
\end{clmproof}

Let \defin{$F_{-1}^\star, F_0^\star, \dots, F_p^\star$} be pairwise vertex-disjoint graphs such that $F_i^\star$ is isomorphic to $F_i$ for all $i\in \set{-1,0,\dots, p}$. Note that unlike $(F_{-1}^\star, F_0^\star, \dots, F_p^\star)$, the graphs in $(F_{-1}, F_0, \dots, F_p)$ may intersect each other. For each $i\in\set{-1, 0,\dots, p}$ let \defin{$\pi_i:V(F_i)\rightarrow V(F_i^\star)$} be an isomorphism from $F_i$ to $F_i^\star$. Let $\mathdefin{F^\star} = F_{-1}^\star\cup \bigcup_{i=0}^{p}F_i^\star$.

Let \defin{$\beta_{-1}$} and \defin{$\beta_{\geq0}$} be minimum-width tree-decompositions of $F_{-1}^\star$ and $\bigcup_{i=0}^{p}F_i^\star$ respectively. Let \defin{$T_{-1}$} and \defin{$T_{\geq0}$} be the underlying trees of $\beta_{-1}$ and $\beta_{\geq0}$ respectively, which we may assume to be vertex-disjoint. Let \defin{$T$} be a tree that contains $T_{-1} \cup T_{\geq0}$ as a spanning subgraph. Then $\mathdefin{\beta}\coloneqq \beta_{-1} \cup \beta_{\geq0}$ is a tree-decomposition of $F^\star$. Now $\bigcup\mathscr{C}-L$ is a forest by \eqref{eq:L-hits-long-cycles-in-G_r}, thus $\beta_{-1}$ has width at most $1$. Furthermore, \cref{thm:birmele}, \eqref{eq:no-long-cycles-in-F_0}, and \eqref{eq:no-long-cycles-in-F_i} together imply that $\beta_{\geq0}$ has width less than $\ell-1$. Therefore for all $t\in V(T)$,
\begin{equation}\label{eq:size-of-bags}
    |\beta(t)| \leq \begin{cases}
        2 & \text{if $t\in V(T_{-1})$,}\\
        \ell-1 & \text{if $t\in V(T_{\geq0})$.}
    \end{cases}
\end{equation}

For subsets $S\subseteq V(F^\star)$, let $\mathdefin{T(S)}\coloneqq T[\set{t\in V(T) : \beta(t)\cap S\not=\emptyset}]$.

An ear $P$ is \defin{$i$-problematic}, for some $i\in [p]$, if it is an $(i,i)$-ear and there exists a cycle in $W_P\cup C_i$ of length less than $\ell$. An ear is \defin{problematic} if it is $i$-problematic for some $i\in [p]$. An ear is \defin{non-problematic} if it is not problematic. For non-problematic admissible ears $P$, let $\mathdefin{\Psi^\star(P)}\coloneqq  \bigcup_{m=-1}^{p}\pi_m(\Psi_m(P))$. Let \defin{$\mathcal{A}$} be the set of all graphs $T(\Psi^\star(P))$ where $P$ is a non-problematic admissible ear. By \cref{clm:components-of-Psi}, $F^\star[\Psi^\star(P)]$ has at most five components for every non-problematic admissible ear $P$. Therefore $\mathcal{A}$ is a collection of subgraphs of $T$ each with at most five components.

Let $\mathdefin{k^\star}$ be the minimum integer such that $2k^\star \geq s(k)+2(k-1)$. Proceed by cases depending on the outcome of \cref{thm:alon} when applied to $T$, $\mathcal{A}$, and $k^\star$.

\textbf{Case \ref{item:packing:alon}} $\mathcal{A}$ has $k^\star$ pairwise vertex-disjoint members: Then there exists a collection \defin{$\mathcal{N}$} of $k^\star$ non-problematic admissible ears such that $(T(\Psi^\star(P)) : P\in \mathcal{N})$ is a collection of pairwise vertex-disjoint graphs. Thus $(\Psi^\star(P):P\in \mathcal{N})$ is a collection of pairwise disjoint sets of vertices of $F^\star$. Consequently, for all distinct $P,Q\in \mathcal{N}$ and all $m\in \{-1, 0, \dots, p\}$ we have $\Psi_m(P)\cap \Psi_m(Q)=\emptyset$. Then by \cref{clm:dist-and-proj} and \cref{clm:dist-and-W}, for all distinct $P,Q\in \mathcal{N}$,
\begin{equation}\label{eq:W_P-and-W_Q-are-far-in-C}
    \dist_{\bigcup\mathscr{C}}(\proj_U(P), \proj_U(Q)) > d\ell.
\end{equation}
and
\begin{equation}\label{eq:W_P-and-W_Q-are-far-in-G}
    \dist_G(V(W_P), V(W_Q)) > d,
\end{equation}

For each $P\in \mathcal{N}$, either $W_P$ is a path or $W_P$ contains a cycle. Let \defin{$\mathcal{P}$} be the set of all $P\in \mathcal{N}$ such that $W_P$ is a path.  

\begin{clm}\label{clm:non-problematic-walks-with-cycles}
$G$ contains a $d$-packing of $|\mathcal{N}\setminus \mathcal{P}|$ cycles each of length at least $\ell$.
\end{clm}

\begin{clmproof}
For each ear $E\in \mathcal{N}\setminus \mathcal{P}$, let $D_P$ be a cycle in $W_E$. Therefore \eqref{eq:W_P-and-W_Q-are-far-in-G} implies that $\set{D_E : E\in \mathcal{N}\setminus \mathcal{P}}$ is a $d$-packing of $|\mathcal{N}\setminus \mathcal{P}|$ cycles in $G$. Now consider each $E\in \mathcal{N}\setminus \mathcal{P}$. Since $W_E$ contains a cycle, $E$ must be an $(i,i)$-ear for some $i\in [p]$. Then since $E$ is non-problematic, every cycle in $W_E\cup C_i$ has length at least $\ell$, implying $\len(D_E)\geq \ell$.
\end{clmproof}

If $|\mathcal{N}\setminus \mathcal{P}| \geq k$, then by \cref{clm:non-problematic-walks-with-cycles} $G$ contains a $d$-packing of $k$ cycles each of length at least $\ell$, a contradiction. Hence $|\mathcal{N}\setminus \mathcal{P}| \leq k-1$, which along with $|\mathcal{N}| = k^\star \geq s(k)/2+k-1$ implies that $2|\mathcal{P}| \geq s(k)$.

Consider the graph 
\[\textstyle \mathdefin{G'}\coloneqq \bigcup\mathscr{C}\cup \bigcup_{E\in \mathcal{P}}W_E.\]
Recall that $\mathscr{C}$ is a collection of pairwise vertex-disjoint cycles. Also, \eqref{eq:W_P-and-W_Q-are-far-in-G} implies that $(W_E : E\in \mathcal{P})$ is a collection of pairwise vertex-disjoint paths. Furthermore, no edge or inner vertex of $W_E$ appears in $\bigcup \mathscr{C}$ for all $E\in \mathcal{P}$. Therefore every vertex $v\in V(G')$ satisfies $\deg_{G'}(v)\in \set{2,3}$, and the degree-3 vertices of $G'$ are precisely the endpoints of paths in $(W_E:E\in \mathcal{P})$. Hence $G'$ has $2|\mathcal{P}| \geq s(k)$ degree-3 vertices.

Any path in $G'$ whose endpoints are degree-$3$ vertices and whose inner vertices are degree-$2$ is a \defin{segment}. Any segment is either a $W_E$ for some $E\in \mathcal{P}$ (called a \defin{$\mathcal{P}$-segment}) or is a path in $\bigcup\mathscr{C}$ between an endpoint of $W_P$ and an endpoint of $W_Q$ for some $P, Q\in \mathcal{P}$ (called a \defin{$\mathscr{C}$-segment}).

\begin{clm}\label{clm:main-cycles-in-G'-are-long}
Every cycle in $G'$ has length at least $\ell$.
\end{clm}

\begin{clmproof}
Every cycle in $\mathscr{C}$ has length at least $\ell$. Now consider any cycle $D$ in $G'$ that is not in $\mathscr{C}$. Then $D$ contains a $\mathcal{P}$-segment. If $D$ contains exactly one $\mathcal{P}$-segment, then there exists $P\in \mathcal{P}$ and $i\in [p]$ such that $P$ is an $(i,i)$-ear and $D \subseteq W_P\cup C_i$. Since $P$ is non-problematic, every cycle in $W_P\cup C_i$ has length at least $\ell$, implying $\len(D)\geq\ell$. On the other hand, if $D$ contains more than one $\mathcal{P}$-segment, then there exists distinct $P,Q\in \mathcal{P}$ and a $\mathscr{C}$-segment $S$ between an endpoint of $W_P$ and an endpoint of $W_Q$ such that $S\subseteq D$. Thus \eqref{eq:W_P-and-W_Q-are-far-in-C} implies $\len(D) > \len(S) \geq \dist_{\bigcup\mathscr{C}}(\proj_U(P), \proj_U(Q)) > d\ell \geq \ell$.
\end{clmproof}

We remark that the following proof is similar to \cref{clm:disjoint-cycles-in-G'-are-far-apart}.

\begin{clm}\label{clm:main-disjoint-cycles-in-G'-are-far-apart}
For every pair of vertex-disjoint cycles $D$ and $D'$ in $G'$, $\dist_G(V(D), V(D')) > d$.
\end{clm}

\begin{clmproof}
Assume for a contradiction that there exists a pair of vertex-disjoint cycles $D$ and $D'$ in $G'$ such that $\dist_G(V(D), V(D'))\leq d$. Consider any $u\in V(D)$ and $u'\in V(D')$ with $\dist_G(u,u') \leq d$.

If $u\in V(W_P)$ and $u'\in V(W_Q)$ for some $P,Q\in \mathcal{P}$ with $W_P\subseteq D$ and $W_Q\subseteq D'$, then since $D$ and $D'$ are vertex-disjoint, $P\not=Q$. Therefore $\dist_G(u,u') \geq \dist_G(V(W_P), V(W_Q)) > d$ by \eqref{eq:W_P-and-W_Q-are-far-in-G}, a contradiction. Hence it may be assumed that $u\not\in V(W_P)$ for every $P\in \mathcal{P}$ with $W_P\subseteq D$. Consequently, there exists $C\in \mathscr{C}$ such that $u\in V(C)$.

Let $P$ be a shortest $(u,u')$-path in $G$. Since $\dist_G(u,u')\leq d$, $\len(P)\leq d$, so $V(P)\subseteq B_G(V(C), d)$. In particular, since $\mathscr{C}$ is a $2d$-packing in $G$, $u'\in B_G(V(C), d)$ and $u'\not\in B_G(V(\bigcup\mathscr{C})\setminus V(C), d)$. Therefore $\proj_U(u')$ exists and is a vertex of $C$. Furthermore, either $u'\in V(C)$ or $u'$ lies in the first or second leg of some $Q\coloneqq q_0 q_1 \cdots q_h\in \mathcal{P}$ with $W_Q\subseteq D'$. In the former case, $\proj_U(u')=u'\in V(D')$, and in the latter case $\proj_U(u')\in V(U\angle{u'}) \subseteq V(U\angle{q_0}\cup U\angle{q_h}) \subseteq V(W_Q) \subseteq V(D')$. Hence $\proj_U(u')\in V(D')$ in both cases.

Let $v_0, v_1, \dots, v_s$ be the sequence of vertices along $P$ such that $v_0=u$, $v_s=u'$, and $s\leq d$. Since $V(P)\subseteq B_G(V(C), d) \subseteq V(G_d)$, then \cref{clm:G_d-has-small-span} implies that $\spn(G_R, \bigcup\mathscr{C}, U, e) = \spn(G_d, \bigcup\mathscr{C}, U[V(G_d)], e) \leq \ell$ for every $e\in E(P)$. Therefore:
\begin{align}
    \dist_C(u,\proj_U(u'))=\dist_{\bigcup\mathscr{C}}(u,\proj_U(u')) &\leq \sum_{i=1}^{s}\dist_{\bigcup\mathscr{C}}(\proj_U(v_{i-1}), \proj_U(v_i))\notag\\
    &= \sum_{i=1}^{s}\textstyle\spn(G_R, \bigcup\mathscr{C}, U, v_{i-1}v_i)\leq s\ell \leq d\ell.\label{eq:main-dist-between-u-and-proj_U(u')}
\end{align}
Now consider a shortest $(u,\proj_U(u'))$-path $S$ in $C$. Recall that $u\in V(D)$ and $\proj_U(u')\in V(D')$. Let $x$ be the first vertex in $S$ starting from $u$ such that $x\in V(W_Q)$ for some $Q\in \mathcal{P}$ with $W_Q\subseteq D$. Let $y$ be the first vertex in $S$ starting from $\proj_U(u')$ such that $y\in V(W_E)$ for some $E\in \mathcal{P}$ with $W_E\subseteq D'$. Since $D$ and $D'$ are vertex-disjoint, both vertices $x$ and $y$ exist and $Q\not=E$. Note that $\dist_C(x,y)\leq \dist_C(u, \proj_U(u')) \leq d\ell$ by \eqref{eq:main-dist-between-u-and-proj_U(u')}. On the other hand, $x\in V(C)\cap V(W_Q) \subseteq \proj_U(Q)$, $y\in V(C)\cap V(W_E)\subseteq \proj_U(E)$, $Q$ and $E$ are distinct elements of $\mathcal{P}\subseteq \mathcal{N}$, and \eqref{eq:W_P-and-W_Q-are-far-in-C} together imply $\dist_C(x,y)\geq \dist_{\bigcup\mathscr{C}}(\proj_U(Q), \proj_U(E)) > d\ell$, a contradiction.
\end{clmproof}

Since $G'$ is a graph with all vertices of degree $2$ or $3$ and contains at least $s(k)$ degree-3 vertices, \cref{thm:simonovits} implies that $G'$ contains $k$ pairwise vertex-disjoint cycles. By \cref{clm:main-cycles-in-G'-are-long}, each of these cycles has length at least $\ell$, and by \cref{clm:main-disjoint-cycles-in-G'-are-far-apart}, these cycles form a $d$-packing in $G$. Hence $G$ contains a $d$-packing of $k$ cycles each of length at least $\ell$, a contradiction.

\textbf{Case \ref{item:covering:alon}} there exists a subset $\mathdefin{X^\star}\subseteq V(T)$ with $|X^\star|\leq 50(k^\star-1)$ such that $X^\star\cap V(A)\not=\emptyset$ for all $A\in \mathcal{A}$: Recall that $\mathcal{A}$ is the set of all graphs $T(\Psi^\star(P))$ where $P$ is a non-problematic admissible ear. Let $\mathdefin{B_{-1}^\star}\coloneqq \bigcup_{t\in X^\star \cap V(T_{-1})}\beta(t)$ and $\mathdefin{B_{\geq0}^\star}\coloneqq \bigcup_{t\in X^\star \cap V(T_{\geq0})}\beta(t)$. Consider any non-problematic admissible ear $P$. The definition of $X^\star$ implies that $(B_{-1}^\star\cup B_{\geq0}^\star)\cap \Psi^\star(P)\not=\emptyset$. Therefore $B_{-1}^\star\cap \pi_{-1}(\Psi_{-1}(P))\not=\emptyset$ or $B_{\geq0}^\star\cap \bigcup_{m=0}^{p}\pi_m(\Psi_m(P))\not=\emptyset$. Let 
$\mathdefin{S_{-1}}\coloneqq \pi_{\!-1}^{-1}(B_{-1}^\star)$ and $\mathdefin{S_{\geq0}}\coloneqq \bigcup_{m=0}^{p}\pi_m^{-1}(B_{\geq0}^\star)$. Then $S_{-1}\subseteq V(\bigcup\mathscr{C})$, $S_{\geq0}\subseteq V(G)$, and $S_{-1}\cap\Psi_{-1}(P)\not=\emptyset$ or $S_{\geq0}\cap \bigcup_{m=0}^{p}\Psi_m(P)\not=\emptyset$. Consequently, by letting $\mathdefin{X_3}\coloneqq  \pad(\bigcup\mathscr{C}, S_{-1}, 2d, \ceil{\ell/2})\cup S_{\geq0}$, \cref{clm:hitting-Phi-and-Psi} implies that
\begin{equation}\label{eq:X_3-hits-non-problematic-admissible-ears}
    \text{$\textstyle B_G(X_3,r)\cap V(P)\not=\emptyset$ for every non-problematic admissible ear $P$.}
\end{equation}
By \cref{cor:def-of-pad},
\begin{align}
    |X_3| &\leq \textstyle |\pad(\bigcup\mathscr{C}, S_{-1}, 2d, \ceil{\ell/2})| + |S_{\geq0}|\notag\\
    &\leq |S_{-1}|(2\ceil{\ell/2}+1) + |S_{\geq0}|.\notag
\intertext{Since $|S_{-1}|\leq |B_{-1}^\star|$ and $|S_{\geq0}|\leq |B_{\geq0}^\star|$, we have}
    |X_3| &\leq |B_{-1}^\star|(2\ceil{\ell/2}+1) + |B_{\geq0}^\star|.\notag
\intertext{By \eqref{eq:size-of-bags},}
    |X_3|&\leq 2|X^\star \cap V(T_{-1})|(2\ceil{\ell/2}+1) + (\ell-1)|X^\star \cap V(T_{\geq0})|.\notag
\intertext{Since $|X^\star \cap V(T_{-1})|+|X^\star \cap V(T_{\geq0})|=|X^\star|\leq 50(k^\star-1)$ and $\ell-1\leq 2\ceil{\ell/2}+1$,}
    |X_3| &\leq (2|X^\star|-|X^\star \cap V(T_{\geq0})|)(2\ceil{\ell/2}+1)\notag\\
    &\leq 100(k^\star-1)(2\ceil{\ell/2}+1).\notag\\
\intertext{Since $2(k^\star-1) \leq s(k)+2(k-1)-1$,}
    |X_3| &\leq 50(s(k)+2(k-1)-1)(2\ceil{\ell/2}+1).\label{eq:size-of-X_3}
\end{align}

\begin{clm}\label{clm:cycles-far-from-M-and-X_3-decompose}
For every cycle $D$ in $G$ of length at least $\ell$, either
\begin{enumerate}
    \item $D$ contains a vertex in $B_G(M\cup X_3,r)$, or\label{item:D-is-close-to-M-or-X_3:cycles-far-from-M-and-X_3-decompose}  
    \item \label{item:decomposition:cycles-far-from-M-and-X_3-decompose} $D$ is the union of pairwise internally disjoint paths $Q_1\cup \cdots \cup Q_{2t}$ with $t\geq 2$, such that
    \begin{itemize}
        \item the start vertex of $Q_i$ is the end vertex of $Q_{i-1 \;\mathrm{mod}\;2t}$ for all $i\in [2t]$;
        \item $V(Q_1\cup Q_3\cup \cdots \cup Q_{2t-1}) = V(D)\cap B_G(V(\bigcup\mathscr{C}), r-d)$;
        \item $\set{Q_2, Q_4, \dots, Q_{2t}}$ is a collection of problematic admissible ears.
    \end{itemize}
\end{enumerate}
\end{clm}

\begin{clmproof}
Consider any cycle $D$ in $G$ of length at least $\ell$. If $D$ contains a vertex in $B_G(M\cup X_3,r)$, then \ref{item:D-is-close-to-M-or-X_3:cycles-far-from-M-and-X_3-decompose} holds. Hence it may be assumed that $V(D)\cap B_G(M\cup X_3,r)=\emptyset$. Recall that $Z\cup L\subseteq M$ by \eqref{eq:M-contains-Z-and-L-and-I}, hence $V(D)\cap Z$ and $V(D)\cap L$ are empty.

As a preliminary step towards \ref{item:decomposition:cycles-far-from-M-and-X_3-decompose}, we show that $V(D)\subseteq (B_G(V(\bigcup\mathscr{C}), r-d)\setminus L)\cup V(F_0^-)$, $V(D)\not \subseteq B_G(V(\bigcup\mathscr{C}), r-d)\setminus L$, and $V(D)\not\subseteq V(F_0^-)$. Then, since $G[B_G(V(\bigcup\mathscr{C}), r-d)\setminus L]$ and $F_0^-$ are vertex-disjoint subgraphs of $G$, it will follow that $D$ is the union of pairwise internally disjoint paths $Q_1\cup \cdots \cup Q_{2t}$ such that
\begin{itemize}
    \item the start vertex of $Q_i$ is the end vertex of $Q_{i-1 \;\mathrm{mod}\; 2t}$ for all $i\in [2t]$;
    \item $V(Q_1\cup Q_3\cup \cdots \cup Q_{2t-1}) = V(D)\cap B_G(V(\bigcup\mathscr{C}), r-d)$;
    \item for all $i\in[t]$, $\len(Q_{2i})\geq 2$ and the inner vertices of $Q_{2i}$ lie in $V(F_0^-)$.
\end{itemize}

Since $V(D)\cap Z=\emptyset$, $V(D)\subseteq B_G(V(\bigcup\mathscr{C}), r-d) \cup V(F_0^-)$. Furthermore, since $V(D)\cap L=\emptyset$, $V(D)\subseteq (B_G(V(\bigcup\mathscr{C}), r-d)\setminus L) \cup V(F_0^-)$. Now \eqref{eq:L-hits-long-cycles-in-G_r} implies $V(D)\not\subseteq B_G(V(\bigcup\mathscr{C}), r-d)\setminus L$, and since $F_0^-$ is an induced subgraph of $G$, \eqref{eq:no-long-cycles-in-F_0} implies $V(D)\not\subseteq V(F_0^-)$. This completes the preliminary step.

To show \ref{item:decomposition:cycles-far-from-M-and-X_3-decompose}, it now suffices to show that $\set{Q_2, Q_4, \dots, Q_{2t}}$ is a collection of problematic admissible ears. Consider any $Q\in \set{Q_2, Q_4, \dots, Q_{2t}}$ and let $Q'$ be the path obtained from $Q$ by deleting its endpoints. Since $V(Q')\subseteq F_0^-$ and $F_0^-$ is an induced subgraph of $G$, $Q'$ is a path in $F_0^-$. Therefore since the endpoints of $Q$ lie in $B_G(V(\bigcup\mathscr{C}), r-d) = \bigcup_{i\in [p]}V(U_i)$, $Q$ is an ear. To see that $Q$ is admissible, note that since $Q\subseteq D$ and $V(D)\cap B_G(M,r)=\emptyset$, $V(Q)\cap B_G(M,r)=\emptyset$, thus $V(W_Q) \subseteq B_G(V(Q), r-d)$ implies that $V(W_Q)\cap B_G(M, d)=\emptyset$. Finally, if $Q$ is non-problematic, then \eqref{eq:X_3-hits-non-problematic-admissible-ears} implies $B_G(X_3,r)\cap V(Q)\not=\emptyset$, therefore $B_G(X_3, r)\cap V(D)\not=\emptyset$, a contradiction.
\end{clmproof}

Let \defin{$\mathcal{E}$} be the set of all pairs $uv\in \binom{V(G)}{2}$ such that there exists a problematic admissible ear whose set of endpoints is $\set{u,v}$. Consider the auxiliary graph $\mathdefin{G_{\Aux}}\coloneqq G_R\cup \mathcal{E}$.

Let \defin{$\mathcal{D}$} be the set of cycles $D$ in $G$ that satisfy $B_G(M\cup X_3, r+d)\cap V(D)=\emptyset$ and $\len(D)\geq \ell$. For each $D\in \mathcal{D}$, let $\mathdefin{H_D}\coloneqq G_{\Aux}[B_{G_{\Aux}}(V(D)\cap B_G(V(\bigcup\mathscr{C}), r-d), \ceil{d/2})]$. Observe that by decomposing $D\in\mathcal{D}$ into a union of paths $Q_1\cup \cdots \cup Q_{2t}$ as per \cref{clm:cycles-far-from-M-and-X_3-decompose}, we may write
\begin{equation}\label{eq:H_D}
    H_D=G_{\Aux}[B_{G_{\Aux}}(V(Q_1\cup Q_3\cup \cdots \cup Q_{2t-1}), \ceil{d/2})].
\end{equation}

\begin{clm}\label{clm:disjoint-H's-implies-far-away-D's}
For every pair $D,D'\in \mathcal{D}$ with $V(H_D)\cap V(H_{D'})=\emptyset$, $\dist_G(V(D), V(D'))>d$.
\end{clm}

\begin{clmproof}
Assume for a contradiction that there exists a pair $D, D'\in\mathcal{D}$ with $V(H_D)\cap V(H_{D'})=\emptyset$ such that $\dist_G(V(D), V(D'))\leq d$. Recall that $B_G(M\cup X_3, r)\cap V(D)=\emptyset$ and $B_G(M\cup X_3, r)\cap V(D')=\emptyset$, hence we may write $D=Q_1\cup \cdots \cup Q_{2t}$ and $D'=Q_1'\cup \cdots \cup Q_{2s}'$ according to \cref{clm:cycles-far-from-M-and-X_3-decompose}.
Let $i\in [2t]$ and $j\in [2s]$ such that $\dist_G(V(Q_i),V(Q_j'))\leq d$. Let $P$ be a shortest $(V(Q_i), V(Q_j'))$-path in $G$.
Let $u$ be the endpoint of $P$ in $Q_i$ and let $u'$ be the endpoint of $P$ in $Q_j'$.

The following shows that if $E$ is a path in $Q_i\cup P\cup Q_j'$ with $\len(E)\geq 2$ and $V(E)\cap B_G(V(\bigcup\mathscr{C}), r-d)$ equals the set of endpoints of $E$, then $E$ is a problematic admissible ear: Since $Q_i\cup Q_j'\subseteq D\cup D'$ and $V(D\cup D')\cap B_G(M\cup X_3, r+d)=\emptyset$, $V(Q_i\cup P\cup Q_j')\cap B_G(M\cup X_3, r)=\emptyset$. Then
\begin{equation}\label{eq:E-is-far-from-M-and-X_3}
    V(E)\cap B_G(M\cup X_3, r)=\emptyset.
\end{equation}
Let $E'$ be the path obtained from $E$ by deleting its endpoints. Then \eqref{eq:E-is-far-from-M-and-X_3} and \eqref{eq:M-contains-Z-and-L-and-I} imply $V(E')\cap Z=\emptyset$, thus $E'$ is a path in $F_0^-$. Therefore since the endpoints of $E$ lie in $B_G(V(\bigcup\mathscr{C}), r-d)$, $E$ is an ear. Next, since $V(W_E)\subseteq B_G(V(E), r-d)$, \eqref{eq:E-is-far-from-M-and-X_3} implies $V(W_E)\cap B_G(M, d)=\emptyset$, thus $E$ is admissible. Therefore \eqref{eq:E-is-far-from-M-and-X_3} and \eqref{eq:X_3-hits-non-problematic-admissible-ears} imply that $E$ is problematic.

Let $V\coloneqq V(Q_1\cup Q_3\cup \cdots \cup Q_{2t-1})$ and $V'\coloneqq V(Q_1'\cup Q_3'\cup \cdots \cup Q_{2s-1}')$. Observe that since $V(H_D)\cap V(H_{D'})=\emptyset$, $\dist_{G_{\Aux}}(V,V') \geq 2\ceil{d/2}+1>d$.

We now show that $i$ is even, $j$ is even, $u$ is an inner vertex of $Q_i$, and $u'$ is an inner vertex of $Q_j'$. By symmetry it suffices to prove that $i$ is even and $u$ is an inner vertex of $Q_i$. Assume for a contradiction that $i$ is odd or $u$ is an endpoint of $Q_i$. Then $u'$ is an inner vertex of $Q_j'$ and $j$ is even, otherwise $d \geq \len(P) \geq \dist_{G_R}(V, V') \geq \dist_{G_{\Aux}}(V,V') > d$, a contradiction. Note that $u\in B_G(V(\bigcup\mathscr{C}), r-d)$ and $u'\not\in B_G(V(\bigcup\mathscr{C}), r-d)$, thus $u\not=u'$. Let $x$ be an endpoint of $Q_j'$. Let $y\in V(P-u')\cap B_G(V(\bigcup\mathscr{C}), r-d)$ such that the path $E\coloneqq x Q_j' u' \cup u' P y$ has minimum possible length. Such a $y$ exists since $u$ is a candidate and $u\not=u'$. Since $Q_j'$ is an ear, the inner vertices of $Q_j'$ are disjoint from $B_G(V(\bigcup \mathscr{C}), r-d)$, therefore $E\cap B_G(V(\bigcup\mathscr{C}), r-d) = \{x,y\}$. Moreover, $\len(E)\geq 2$ since $E$ has an inner vertex $u'$. Consequently, $E$ is a problematic admissible ear, and $xy$ is an edge of $G_{\Aux}$. Then $xy\cup yPu$ is a $(V',V)$-path in $G_{\Aux}$ of length at most $d$, a contradiction. Hence $i$ is even, $j$ is even, $u$ is an inner vertex of $Q_i$, and $u'$ is an inner vertex of $Q_j'$.

For the final contradiction of the claim, we show that there exists edges $e,e'\in E(G_{\Aux})$, such that $e$ is between an endpoint of $Q_i$ and an inner vertex of $P$, and $e'$ is between an endpoint of $Q_j$ and an inner vertex of $P$. This will imply $d\geq \dist_{G_{\Aux}}(V, V')$, a contradiction. By symmetry, it suffices to prove the existence of $e$. Let $Q$ be a shortest path in $Q_i\cup P\cup Q_j'$ between the set of endpoints of $Q_i$ and the set of endpoints of $Q_j'$. Then the endpoints of $Q$ lie in $B_G(V(\bigcup\mathscr{C}), r-d)$. If $\len(Q)\leq 1$, then $d\geq 1\geq \dist_{G_R}(V, V') \geq \dist_{G_{\Aux}}(V, V')$, a contradiction. Hence $\len(Q)\geq 2$. If $V(Q)\cap B_G(V(\bigcup\mathscr{C}), r-d)$ equals the set of endpoints of $Q$, then $Q$ is a problematic admissible ear and there exists an edge of $G_{\Aux}$ between the endpoints of $Q$, which implies $d\geq 1 \geq \dist_{G_{\Aux}}(V, V')$, a contradiction. Hence it may be assumed that some inner vertex of $Q$ lies in $B_G(V(\bigcup\mathscr{C}), r-d)$. Since $Q\subseteq Q_i\cup P\cup Q_j'$ and $Q_i$ and $Q_j'$ are ears, such a vertex lies in $V(P)$. Let $x$ be the endpoint of $Q$ in $Q_i$. Let $y\in V(P)\cap B_G(V(\bigcup\mathscr{C}), r-d)$ such that the path $E\coloneqq x Q u \cup u P y$ has minimum possible length. Since the inner vertices of $Q_i$ and $Q_j'$ are disjoint from $B_G(V(\bigcup\mathscr{C}), r-d)$, $u$ and $u'$ are not in $B_G(V(\bigcup\mathscr{C}), r-d)$, thus $y$ is an inner vertex of $P$. Furthermore, $E\cap B_G(V(\bigcup\mathscr{C}), r-d) = \{x,y\}$. Also, $\len(E)\geq 2$ since $E$ has an inner vertex $u$. It follows that $E$ is a problematic admissible ear, and $e\coloneqq xy$ is an edge of $G_{\Aux}$ between an endpoint of $Q_i$ and an inner vertex of $P$, as required.
\end{clmproof}

\begin{clm}\label{clm:U_Aux-is-BFS-spanning}
$U$ is a $\bigcup\mathscr{C}$-supported BFS-spanning subgraph of $G_{\Aux}$.
\end{clm}

\begin{clmproof}
Recall that for every edge $e\in \mathcal{E}$, both endpoints of $e$ have distance exactly $r-d$ from $V(\bigcup\mathscr{C})$ in $G$. Hence  $\dist_{G\cup \mathcal{E}}(V(\bigcup\mathscr{C}), v)=\dist_G(V(\bigcup\mathscr{C}), v)$ for all $v\in V(G\cup \mathcal{E})=V(G)$. Consequently, every $\bigcup\mathscr{C}$-supported BFS-spanning subgraph of $G[B_G(V(\bigcup\mathscr{C}), R)]$ is a $\bigcup\mathscr{C}$-supported BFS-spanning subgraph of $(G\cup \mathcal{E})[B_{G\cup \mathcal{E}}(V(\bigcup\mathscr{C}), R)]$. Therefore since $G[B_G(V(\bigcup\mathscr{C}), R)] = G_R$ and $(G\cup \mathcal{E})[B_{G\cup \mathcal{E}}(V(\bigcup\mathscr{C}), R)] = G_R\cup \mathcal{E} = G_{\Aux}$, $U$ is a $\bigcup\mathscr{C}$-supported BFS-spanning subgraph of $G_{\Aux}$.
\end{clmproof}

\begin{clm}\label{clm:main-edges-of-H_D-have-big-span}
For every $D\in \mathcal{D}$, $H_D$ is a non-null connected subgraph of $G_{\Aux}$ such that every edge $e\in E(H_D)$ satisfies $\spn(G_{\Aux}, \bigcup\mathscr{C}, U, e)\leq \ell$.
\end{clm}

\begin{clmproof}
Consider any $D\in \mathcal{D}$ and write $D=Q_1\cup \cdots \cup Q_{2t}$ as per \cref{clm:cycles-far-from-M-and-X_3-decompose}. First note that $\emptyset \not= V(Q_1\cup Q_3\cup \cdots \cup Q_{2t-1}) = V(D)\cap B_G(V(\bigcup\mathscr{C}), r-d)\subseteq V(G_R) = V(G_{\Aux})$. Now since $\{Q_2, Q_4, \dots, Q_{2t}\}$ is a collection of problematic admissible ears, for each $i\in [t]$ there exists an edge in $G_{\Aux}$ between the endpoints of $Q_{2i}$. Consequently, $V(Q_1\cup Q_3\cup \cdots \cup Q_{2t-1})$ is a non-empty connected set of vertices in $G_{\Aux}$. Then by \eqref{eq:H_D}, $H_D$ is a non-null connected subgraph of $G_{\Aux}$. We now show that $\spn(G_{\Aux}, \bigcup\mathscr{C}, U, e)\leq \ell$ for every edge $e\in E(H_D)$. Consider any edge $e\in E(H_D)$. Since $H_D\subseteq G_{\Aux} = G_R\cup \mathcal{E}$, either \textbf{(1)} $e\in \mathcal{E}$ or \textbf{(2)} $e\in E(G_R)$.

\textbf{Case (1)} $e\in \mathcal{E}$: Then for some $i\in [p]$ there exists an $i$-problematic admissible ear $E$ that has the same endpoints as $e$. Then $E$ is an $(i,i)$-ear and there exist a cycle in $W_E\cup C_i$ of length less than $\ell$. Therefore since $\len(C_i)\geq \ell$, the endpoints of the walk $W_E$ have distance less than $\ell$ in $C_i$, which implies $\spn(G_{\Aux}, \bigcup\mathscr{C}, U, e)<\ell$ as required.

\textbf{Case (2)} $e\in E(G_R)$: Let $V\coloneqq V(D)\cap B_G(V(\bigcup\mathscr{C}), r-d)$. Recall that $W\subseteq M$ by \eqref{eq:W-and-intersections-in-X_1-ball} and \eqref{eq:def-of-M}. We show that an endpoint of $e$ is not in $W$, which implies that $\spn(G_{\Aux}, \bigcup\mathscr{C}, U, e) = \spn(G_R, \bigcup\mathscr{C}, U, e)\leq \ell$ as required. Recall that $e\in E(H_D)$. Let $P$ be a shortest path in $G_{\Aux}$ between $V$ and the set of endpoints of $e$. Then $\len(P)\leq \ceil{d/2}$ and $e\not\in E(P)$. Let $v$ be the endpoint of $e$ in $P$. If $P\subseteq G_R$, then $B_G(M, r+d)\cap V(D)=\emptyset$ implies that $v$ is not in $B_G(M, r+d-\len(P))$. Therefore $v\not\in W$ as required. On the other hand $P\not\subseteq G_R$, so $P$ contains an edge from $\mathcal{E}$. Let $xy\in E(P)\cap \mathcal{E}$ such that $\dist_P(v, \set{x,y})$ is minimum. Then $\dist_G(v, \set{x,y}) \leq \ceil{d/2}$. Let $E$ be a problematic admissible ear whose set of endpoints is $\set{x,y}$. Since $E$ is admissible, $V(W_E)\cap B_G(M, d)=\emptyset$, thus $\set{x,y}\subseteq V(W_E)$ implies that $\set{x,y}\cap B_G(W, d)=\emptyset$. Therefore since $\dist_G(v, \set{x,y})\leq \ceil{d/2}$, $v\not\in W$ as required. 
\end{clmproof}

Now $\mathscr{C}$ is a non-empty collection of pairwise vertex-disjoint cycles in $G\cup\mathcal{E}$ and, by \cref{clm:U_Aux-is-BFS-spanning}, $U$ is a $\bigcup\mathscr{C}$-supported BFS-spanning subgraph of $G_{\Aux}$. 
Note that the earlier definition of $G_{\Aux}$ is equivalent to  $(G\cup \mathcal{E})[B_{G\cup \mathcal{E}}(V(\bigcup\mathscr{C}), R)]$.
Furthermore by \cref{clm:main-edges-of-H_D-have-big-span}, $\mathdefin{\mathcal{H}}\coloneqq \set{H_D:D\in \mathcal{D}}$ is a collection of non-null connected subgraphs of $G_{\Aux}$ such that $\spn(G_{\Aux}, \bigcup\mathscr{C}, U, e)\leq \ell$ for every $e\in E(\bigcup\mathcal{H})$. Hence we may proceed by cases depending on the outcome of \cref{lem:cycle-helly-for-small-span-collection}.

\textbf{Case \ref{item:packing:cycle-helly-for-small-span-collection}} $\mathcal{H}$ contains $k$ pairwise vertex-disjoint members: Then there exists $\mathcal{D}'\subseteq \mathcal{D}$ with $|\mathcal{D}'|=k$ such that $V(H_A)\cap V(H_B)=\emptyset$ for all distinct $A,B\in \mathcal{D}'$. By \cref{clm:disjoint-H's-implies-far-away-D's}, $\mathcal{D}'$ is a $d$-packing of $k$ cycles in $G$, and by definition of $\mathcal{D}$, each of these cycles has length at least $\ell$. Hence $G$ contains a $d$-packing of $k$ cycles each of length at least $\ell$, a contradiction.

\textbf{Case \ref{item:alternative:cycle-helly-for-small-span-collection}} there exists $\mathdefin{Y}\subseteq V(\bigcup\mathscr{C})$ with $|Y|\leq k-1+|\mathscr{C}|$ such that for $\mathdefin{X_4'}\coloneqq B_{\bigcup\mathscr{C}}(Y,\floor{\ell/2})$, $B_U(X_4', R)\cap V(H)\not=\emptyset$ for all $H\in \mathcal{H}$.
Since $|\mathscr{C}|\leq p\leq k-1$, $|Y|\leq 2(k-1)$. Let $\mathdefin{X_4}\coloneqq \pad(\bigcup\mathscr{C}, Y, 2\ceil{d/2}+1, \ceil{\ell/2})$. Then by \cref{cor:def-of-pad}, 
\begin{equation}\label{eq:size-of-X_4}
    |X_4|\leq |Y|(2\ceil{\ell/2}+1) \leq 2(k-1)(2\ceil{\ell/2}+1).
\end{equation}
Let $\mathdefin{X}\coloneqq X_0\cup X_1 \cup X_2 \cup X_3 \cup X_4$. Then (\ref{eq:size-of-X_0}, \ref{eq:size-of-X_1}, \ref{eq:size-of-X_2}, \ref{eq:size-of-X_3}, \ref{eq:size-of-X_4}) imply
\begin{align}
    |X| &\leq |X_0| + |X_1| + |X_2| + |X_3| + |X_4|\notag\\
    &\leq k-1 + (s(k)-1)(2\ceil{\ell/2}+1) + 2(k-1)(2\ceil{\ell/2}+1)\notag\\
    &\qquad\qquad+ 50(s(k)+2(k-1)-1)(2\ceil{\ell/2}+1) + 2(k-1)(2\ceil{\ell/2}+1)\notag\\
    &= (51s(k)+104k-155)(2\ceil{\ell/2}+1)+k-1\notag\\
    &=f(k,\ell).\label{eq:f(k)}
\end{align}

By \eqref{eq:def-of-M},
\begin{align}
    B_G(M\cup X_3, r+d)\cup B_G(X_4, r-\floor{d/2}) &\subseteq B_G(X, R+2r+2d)\qquad\qquad\notag\\
    &\subseteq B_G(X, 21d)\notag\\
    &= B_G(X, g(d)).\label{eq:g(d)}
\end{align}

\begin{clm}\label{clm:main-hitting}
$G-B_G(X, g(d))$ has no cycle of length at least $\ell$.
\end{clm}

\begin{clmproof}
Consider any cycle $D$ in $G$ of length at least $\ell$. We show that $B_G(X, g(d)) \cap V(D)\not=\emptyset$. If $B_G(M\cup X_3, r+d)\cap V(D)\not=\emptyset$, then \eqref{eq:g(d)} implies $B_G(X, g(d)) \cap V(D)\not=\emptyset$ as required. Hence it may be assumed that $B_G(M\cup X_3, r+d)\cap V(D)=\emptyset$, in other words, $D\in \mathcal{D}$. Let $u \in B_U(X_4', R)\cap V(H_D)$ as promised by the case \ref{item:alternative:cycle-helly-for-small-span-collection} above. Since $X_4'\subseteq V(\bigcup\mathscr{C})$, there exists $x\in X_4'$ such that $x=\proj_U(u)$. Since $u\in V(H_D) = B_{G_{\Aux}}(V(D)\cap B_G(V(\bigcup\mathscr{C}), r-d), \ceil{d/2})$, there exists $v\in V(D)\cap B_G(V(\bigcup\mathscr{C}), r-d)$
and a path $q_0 q_1\cdots q_s$ in $H_D$ such that $q_0=v$, $q_s=u$, and $s\leq \ceil{d/2}$. Then by \cref{clm:main-edges-of-H_D-have-big-span},
\begin{align*}
    \dist_{\bigcup\mathscr{C}}(\proj_U(v), x) &\leq \sum_{i=1}^{s}\dist_{\bigcup\mathscr{C}}(\proj_U(q_{i-1}), \proj_U(q_i))\\
    &= \sum_{i=1}^{s}\textstyle \spn(G_{\Aux}, \bigcup\mathscr{C}, U, q_{i-1}q_i) \leq s\ell \leq \ceil{d/2}\ell.
\end{align*}
Furthermore, since $x\in X_4'=B_{\bigcup\mathscr{C}}(Y, \floor{\ell/2})$,
\[\dist_{\bigcup\mathscr{C}}(\proj_U(v), Y) \leq \dist_{\bigcup\mathscr{C}}(\proj_U(v), x) + \dist_{\bigcup\mathscr{C}}(x, Y) \leq \ceil{d/2}\ell + \floor{\ell/2} \leq (2\ceil{d/2}+1)\ceil{\ell/2}.\]
Since $X_4=\pad(\bigcup\mathscr{C}, Y, 2\ceil{d/2}+1, \ceil{\ell/2})$, \cref{cor:def-of-pad} implies $\dist_{\bigcup\mathscr{C}}(\proj_U(v), X_4)\leq \ceil{d/2}$. Therefore since $v\in B_G(V(\bigcup\mathscr{C}), r-d)$, $v\in B_U(\proj_U(v), r-d)$, thus $v\in B_{\bigcup\mathscr{C}\cup U}(X_4, r-\floor{d/2})$. Recall that $v\in V(D)$, thus we have shown that $D$ has a vertex in $B_G(X_4, r-\floor{d/2})$. Then by \eqref{eq:g(d)}, $B_G(X, g(d))\cap V(D)\not=\emptyset$, as desired.
\end{clmproof}

\cref{clm:main-hitting} and \eqref{eq:f(k)} imply that the theorem holds.
\end{proof}

\section*{Acknowledgments}
Part of this work was carried out during the following workshops:  The Graph Theory Workshop held in January 2025 in Oberwolfach (Germany);  the Twelve Annual Workshop on Geometry and Graphs held in February 2025 at the Bellairs Research Institute of McGill University (Barbados);  the second Belgian Graph Theory Conference held in July 2025 in Brussels (Belgium);  the Thirteenth Annual Workshop on Geometry and Graphs held in February 2026 at the Bellairs Research Institute of McGill University (Barbados);   the Focused Workshop on Erd\H{o}s--P\'osa problems held in March 2026 in Będlewo (Poland); and the MATRIX workshop ``Global Structure and Geometry of Graphs'' held in April 2026 in Creswick (Australia). The authors are thankful to all the organizers and participants for providing a stimulating research environment. The fourth author would also like to thank David Wood for his feedback on the presentation of the paper.

\bibliographystyle{plainurlnat}

\setlength{\bibsep}{0.1ex}

\bibliography{bibliography}

\end{document}